\documentclass[11pt]{article}
\pdfoutput=1

\def\pb{}

\usepackage[dvipsnames]{xcolor}

\def\beq{\begin{equation} }\def\eeq{\end{equation} }\def\1{\mathbf{1}}

\usepackage[framemethod=default]{mdframed}
\usepackage{caption}

\usepackage{indentfirst}
\usepackage{bm, mathrsfs, graphics,float,amssymb,amsmath,subeqnarray,setspace,graphicx,amsthm,epstopdf,subfigure, enumerate, color}
\usepackage[utf8]{inputenc}
\usepackage[colorlinks,
      linkcolor=red,
      anchorcolor=blue,
      citecolor=blue
      ]{hyperref}
\usepackage{natbib}
\usepackage{fullpage}
\numberwithin{equation}{section}

\newtheorem{lemma}{Lemma}
\newtheorem{theorem}{Theorem}

\newtheorem{definition}{Definition}
\newtheorem{corollary}[theorem]{Corollary}
\newtheorem{remark}{Remark}

\ifx\assumption\undefined
\newtheorem{assumption}{Assumption}
\fi

\newcommand{\cO}{\mathcal{O}}

\newcommand{\EE}{\mathbb{E}}
\newcommand{\RR}{\mathbb{R}}

\newcommand{\bA}{\bm{A}}
\newcommand{\bS}{\bm{S}}

\newcommand{\bH}{\bm{H}}
\newcommand{\bI}{\bm{I}}
\newcommand{\bD}{\bm{D}}
\newcommand{\bP}{\bm{P}}

\newcommand{\PP}{\mathbb{P}}

\def\BR{{\mathbb{R}}}
\newcommand{\x}{\bm{x}}
\newcommand{\s}{\bm{s}}
\newcommand{\vx}{\bm{x}}
\newcommand{\vy}{\bm{y}}
\newcommand{\y}{\bm{y}}
\newcommand{\si}{\bm{s}_i}
\newcommand{\g}{\mathbf{g}}
\newcommand{\bv}{\bm{v}}

\def\inner#1#2{\langle #1, #2 \rangle}
\usepackage{bbm}

\newcommand{\hE}{\hat{\cE}}

\def\cK{\mathcal{K}}

\newcommand{\cN}{\mathcal{N}}

\newcommand{\hL}{\hat{L}}

\newcommand{\hv}{\hat{\bv}}
\newcommand{\hg}{\hat{\g}}

\newcommand{\hgam}{\hat{\gamma}}
\newcommand{\hzt}{\hat{\zeta}}
\newcommand{\cE}{\mathcal{E}}

\usepackage{multirow}
\usepackage{tablefootnote}
\usepackage{colortbl}
\usepackage{hhline}

\usepackage{algorithm}
\usepackage{algorithmic}

\begin{document}
\title{
Why Does Adaptive Zeroth-Order Optimization Work?
}

\author{
	Haishan Ye
	\thanks{
		Xi'an Jiaotong University;
		email: hsye\_cs@outlook.com
	}
	\and
	Luo Luo
	\thanks{
		Fudan University;
		email: luoluo@gmail.com
	}
}
\date{
	\today}

\maketitle

\def\RB{\RR}
\def\TH{\tilde{H}}
\newcommand{\ti}[1]{\tilde{#1}}
\def\diag{\mathrm{diag}}
\newcommand{\norm}[1]{\left\|#1\right\|}
\newcommand{\dotprod}[1]{\left\langle #1\right\rangle}
\def\EB{\EE}
\def\tr{\mathrm{tr}}

\begin{abstract}
	Zeroth-order (ZO) optimization is popular in real-world applications that accessing the gradient information is expensive or unavailable. 
    Recently, adaptive ZO methods that normalize gradient estimators by the empirical standard deviation of function values have achieved strong practical performance, particularly in  fine-tuning the large language model. 
    However, the theoretical understanding of such strategy remains limited. 
    In this work, we show that the empirical standard deviation is, with high probability, closely proportional to the norm of the (stochastic) gradient. 
    Based on this insight, we analyze adaptive ZO methods under the generalized $(L_0,L_1)$-smoothness condition with respect to the matrix norm. 
    We establish explicit convergence rates and query complexity bounds for both deterministic and stochastic settings, demonstrating that adaptive ZO methods achieve the faster convergence and the improved query efficiency compared to the vanilla ZO methods with fixed-step.
\end{abstract}


\pb\section{Introduction}

Zeroth-order (ZO) optimization has become a fundamental pillar of modern machine learning.
It is widely used in a variety of applications where gradient information is unavailable or prohibitively expensive to obtain, including black-box adversarial attacks \citep{ilyas18a,chen2017zoo}, reinforcement learning \citep{chrabaszcz2018back,salimans2017evolution,mania2018simple}, memory-efficient large language model (LLM) fine-tuning \citep{malladi2023fine,dang2025fzoo}, etc.
The vanilla ZO algorithm estimates the gradient using a batch of function value queries. Specifically, the gradient estimator is constructed as
\begin{equation*}
	\g(\x) = \sum_{i=1}^\ell \frac{f(\x+\alpha \si) - f(\x)}{\alpha}\cdot \si,
\end{equation*}
where $f(\x)$ is the objective function, $\ell$ is the batch size, $\alpha>0$ is the smoothing parameter, and $\si \sim \cN(\bm{0}, \bm{I}_d / \ell)$ with $\bm{I}_d$ being the $d\times d$ identity matrix.

When minimizing a stochastic objective, where only noisy function evaluations $f(\x;\;\xi)$ are accessible, the corresponding stochastic gradient estimator is given by
\begin{equation*}
	\hg(\x) = \sum_{i=1}^\ell \frac{f(\x+\alpha \si;\;\xi) - f(\x;\;\xi)}{\alpha} \cdot \si.
\end{equation*}
Using either estimator, the vanilla ZO algorithm \citep{Nesterov2017,ghadimi2013stochastic} updates the iterate according to
\begin{equation*}
	\x_{t+1} = \x_t - \eta \g(\x_t) \qquad\mbox{ or }\qquad
	\x_{t+1} = \x_t - \eta \hg(\x_t),
\end{equation*}
where $0<\eta$ denotes the step size.
This vanilla ZO scheme has been successfully applied to black-box adversarial attacks \citep{ilyas18a}, and  memory-efficient LLM tuning \citep{malladi2023fine}.

To further improve the performance of vanilla ZO methods, a popular enhancement is to adaptively control the step size by normalizing the estimated gradient using the standard deviation $\sigma_t$ of the $\ell$ sampled function values (see Eq.~\eqref{eq:sig} and Eq.~\eqref{eq:sig1}). The resulting update rule takes the form
\begin{equation}\label{eq:norm}
	\x_{t+1} = \x_t - \eta \cdot\frac{\g(\x_t)}{\sigma_t} \qquad\mbox{ or }\qquad
	\x_{t+1} = \x_t - \eta \cdot\frac{\hg(\x_t)}{\sigma_t}.
\end{equation}
The intuition behind this normalization is that a small standard deviation $\sigma_t$ indicates that the current iterate $\x_t$ lies in a relatively flat region of the landscape, and therefore a larger effective step size is desirable.

Empirically, this adaptive ZO strategy achieves substantial improvements over the vanilla ZO method, particularly in large-scale language model applications. For instance, \citet{dang2025fzoo} apply the adaptive ZO update in Eq.~\eqref{eq:norm} to LLM fine-tuning and demonstrate significantly higher query efficiency than MeZO \citep{malladi2023fine}, which adopts a fixed-step vanilla ZO method. More recently, adaptive ZO methods have also shown promising performance in reinforcement learning from human feedback \citep{qiu2025evolution}.

Despite these empirical successes, a rigorous theoretical understanding of why adaptive ZO methods outperform vanilla ZO methods remains largely unexplored. In this paper, we aim to bridge this gap by providing a principled theoretical explanation for the advantages of adaptive ZO methods. We summarize our main contribution as follows.
\begin{itemize}
\item First, we show that the standard deviation $\sigma_t$ is, with high probability, nearly proportional to the norm of the (stochastic) gradient. This result formally justifies the intuition that a small $\sigma_t$ corresponds to a flat landscape around the current iterate.
\item Second, we analyze the convergence rate and query complexity of adaptive ZO methods for generalized $(L_0,L_1)$-smooth functions with respect to the $\bH$-norm, where $\bH$ is a $d\times d$ positive semi-definite matrix, rather than the standard Euclidean norm. 
The notion of original $(L_0,L_1)$-smoothness was first introduced by \citet{zhang2019gradient} to demonstrate the superiority of gradient clipping in real-world applications. 
In this work, we further generalize this setting to accommodate the $\bH$-norm smoothness.
\item Finally, we establish explicit convergence rates and query complexity bounds for adaptive ZO methods in both deterministic and stochastic settings on the generalized $(L_0,L_1)$-smooth functions. 
Our results demonstrate that adaptive ZO methods can \emph{converge arbitrarily faster than fixed-step vanilla ZO methods}. Consequently, the query complexity of adaptive ZO methods can be \emph{arbitrarily smaller than that of fixed-step vanilla ZO methods}. A detailed discussion of these improvements is provided in Remark~\ref{remark:deter-M} and Remark~\ref{rmk:sto_M}. 
\end{itemize}

\subsection{Related Works}

\paragraph{Adaptive Zeroth-Order Method.} 
The adaptive ZO method rediscovered independently times and times.
\citet{salimans2017evolution} use a similar normalization as the one in Eq.~\eqref{eq:norm} and call it as the virtual batch normalization of the neural network policies.
\citet{mania2018simple} improve the work of \citet{salimans2017evolution} and add second-order information to achieve faster convergence rate.
\citet{lyu2019black} propose an algorithm called INGO which is similar to the adaptive ZO method but with second-order information.
Furthermore, \citet{lyu2019black} also show that INGO will converge a stationary point even if the objective function is nonconvex.
Recently, \citet{dang2025fzoo} propose FZOO which is also a kind of adaptive ZO method but with  Rademacher random vectors as the search directions
instead of the Gaussian vectors. 
\citet{dang2025fzoo} first realize that the standard deviation $\sigma_t$ is related to the (stochastic) gradient norm and show that $\EE\left[\sigma_t^2\right]$ is almost proportional to the square of (stochastic) gradient norm.
However, the property of $\EE\left[\sigma_t^2\right]$ can \emph{not} directly derive that $\sigma_t$ should be proportional to the (stochastic) gradient norm. 
Furthermore, \citet{dang2025fzoo} only show that FZOO will converge to a stationary point but  \emph{without} a sound theory to show FZOO will outperform vanilla ZO method. 

{\paragraph{Relaxed Smoothness.}

The empirical observations in training neural networks indicate that the relaxed smoothness can better characterize the objective function  than the standard smoothness \citep{zhang2019gradient,crawshaw2022robustness,cooper2024empirical}.
The (stochastic) first-order methods for optimization problems under relaxed smoothness conditions have been extensively studied in recent years \citep{chen2023generalized,vankov2024optimizing,li2024problem,li2023convergence,gaash2025convergence,gorbunov2024methods,koloskova2023revisiting,chezhegov2025convergence,tovmasyan2025revisiting,crawshaw2025complexity,lobanov2024linear,tyurin2024toward,reisizadeh2025variance,zhang2020improved,yu2025mirror,khirirat2024error,khirirat2024communication,jiang2025decentralized,borodich2025nesterov}.
Notably, incorporating the step of normalization or clipping into stochastic gradient descent (SGD) can efficiently achieve  approximate stationary points of  relaxed smooth functions, which cloud be arbitrarily faster than the vanilla SGD with the fixed stepsize when the initialization is poor \citep{zhang2019gradient}. 
Additionally, \citet{xie2024trust} introduce the trust-region methods to further improve the convergence by additional assuming the second-order smoothness condition.
Recently, \citet{lobanov2025power} apply the uniform smoothing to establish the zeroth-order variants of clipped and normalized gradient descent, achieving the linear convergence rates for convex objectives with sufficient large batch sizes. 
}


\section{Preliminaries}

In this section, we first formalization the notations and our problem settings, then introduce the background of zeroth-order optimization.

\subsection{Notations}

We introduce notations used throughout this paper. 
A symmetric matrix $\bA \in \RR^{d \times d}$ is called positive semi-definite if it holds that $\x^\top \bA \x \geq 0$ for any non-zero vector $\x \in \RR^d$.
We let~$\bA = \bm{U} \bm{\Lambda} \bm{U}^\top $ be the eigenvalue decomposition of positive definite matrix $\bA$, where $\bm{U}$ is an orthonormal matrix which contains the eigenvectors of $\bA$ and $\bm{\Lambda} = \diag(\lambda_1,\lambda_2,\dots, \lambda_d)$ is a diagonal matrix with $\lambda_1 \ge \lambda_2\ge\dots\ge \lambda_d >0$. 
This paper uses $\tr(\bA) = \sum_{i=1}^d A_{ii}$ to denote the trace of the matrix~$\bA$.
Given two vectors $\x, \y\in\RR^d$, we use $\dotprod{\x,\y} = \sum_{i=1}^{d} x_i y_i$ to denote their inner product.
Given symmetric positive semi-definite matrix $\bA$, we let $\|\bA\|_{F} \triangleq (\sum_{i,j}A_{ij}^{2})^{1/2}=(\sum_{i}\lambda_{i}^{2})^{1/2}$
be its Frobenius norm and
$\|\bA\|\triangleq \lambda_{1}$ be its the spectral norm. 
For given vector $\x\in\RR^d$, we define its Euclidean norm and $\bA$-norm as $\norm{\x} \triangleq \sqrt{\sum_{i=1}^d x_i^2}$ and 
$\|\x\|_{\bA} \triangleq \sqrt{\x^\top\bA\x}$, respectively.

\subsection{Problem Settings}

First, we  consider the following deterministic minimization problem
\begin{equation}\label{eq:obj0}
	\min_{\x\in\RR^d} f(\x),
\end{equation}
where the ZO algorithms can access the function value $f(\x)$ for given $\x\in\BR^d$. 
We also study the stochastic formulation
\begin{equation}\label{eq:obj1}
	\min_{\x\in\RR^d} f(\x) = \EE_\xi\left[f(\x;\xi)\right],
\end{equation} 
where the ZO algorithms can draw the random variable $\xi$ to access the stochastic function value evaluation for given $\x\in\BR^d$.

Next, we impose assumptions for above optimization problems.

\begin{assumption}\label{ass:LL}
	We suppose the differentiable function $f:\BR^d\to\BR$ is $(L_0, L_1)$-smooth with respect to a positive semi-definite matrix $\bH\in\BR^{d\times d}$, i.e., if for all $\x,\y \in\RR^d$ with $\norm{\y - \x}_{\bH} \leq 1/{(L_1\sqrt{\norm{\bH}})}$ holds that
	\begin{equation}\label{eq:LL}
	 \Big|	f(\y) - f(\x) - \dotprod{\nabla f(\x), \y-\x} \Big| \leq \frac{L_0 + L_1 \norm{\nabla f(\x)}}{2} \norm{\y - \x}_{\bH}^2.
	\end{equation}
\end{assumption}

\begin{assumption}\label{ass:lower-bounded}
	We suppose the  objective $f:\BR^d\to\BR$ is lower bounded, i.e., it holds~$f^*=\inf_{\x\in\BR^d}f(\x)>-\infty$.
\end{assumption}

\begin{assumption}[{\citet{zhang2019gradient}}]\label{ass:var}
We suppose  the stochastic gradient $\nabla f(\x;\xi)$ in formulation (\ref{eq:obj1}) has bounded deviation, i.e., it holds $\norm{\nabla f(\x;\;\xi) - \nabla f(\x)} \leq \nu$ for all $\x\in\BR^d$, where $\nu > 0 $ is a constant.
\end{assumption}

Note that Assumption~\ref{ass:LL} is a generalization of $(L_0, L_1)$-smooth introduced by \citet{zhang2020improved}, since we use the $\bH$-norm to describe the distance in this condition.
Clearly, Assumption~\ref{ass:LL} will reduce to standard $(L_0, L_1)$-smooth by taking $\bH$ be the identity matrix.
In fact, Assumption~\ref{ass:LL} is important to the analysis of the adaptive ZO method, and it is helpful to show the adaptive ZO method can achieve the much faster convergence rate than the vanilla ZO method.


\subsection{Oblivious Random Sketching Matrix}
\label{subsec:Sketch}

In this subsection, we introduce the notion of oblivious random sketching matrix.
Note that the recent work \citep{ye2025unified} showed that the search directions of ZO methods drawn from random sketching matrices can help ZO methods achieve weak dimension dependent query complexity (refer to Remark~\ref{rmk:Q} and Remark~\ref{rmk:Q1}). 
\begin{definition}[Sketching in Matrix Product]\label{def:ske}
	Given a matrix $\bA \in \mathbb{R}^{d \times d}$, parameters $\gamma\in(0,1]$, $k \geq 1$, and $\delta > 0$, we say a matrix $\bS \in \mathbb{R}^{d \times \ell}$, independent of $\bA$, is an oblivious $(\gamma, k, \delta)$-random sketching matrix for matrix product if the following holds with probability at least $1 - \delta$:
	\begin{equation}\label{eq:mat_pd}
		\norm{\bA^\top \bS\bS^\top \bA - \bA^\top \bA} \le \gamma \left(\norm{\bA}^2 + \frac{\norm{\bA}_F^2}{k}\right) .
	\end{equation} 
\end{definition}

\begin{itemize}
\item \textbf{Gaussian sketching matrix:} The most classical sketching matrix is the Gaussian sketching matrix $\bS\in\RR^{\ell\times d}$, whose entries are i.i.d.\ from the univariate Gaussian distribution of mean 0 and variance $1/\ell$. 
Owing to the well-known concentration properties \citep{woodruff2014sketching},  Gaussian random matrices are very attractive. 
It is only requires $\ell = \Omega\left((k + \log({1}/{\delta}))\gamma^{-2}\right)$ to achieve the  $(\gamma, k, \delta)$-approximate matrix product property defined in Definition~\ref{def:ske} \citep{cohen2016optimal}.

\item \textbf{Rademacher Sketching:} As a discrete alternative, the Rademacher matrix uses i.i.d. entries from $\{ \pm 1/\sqrt{\ell} \}$. It shares nearly identical sample complexity bounds with Gaussian matrices but is more efficient in terms of random bit consumption. Recent work has also demonstrated its efficacy in accelerating zeroth-order optimization for fine-tuning large language models \citep{dang2025fzoo}.

\item \textbf{SRHT:} Structured sketching via SRHT, formulated as $\bS = \bP\bH_d\bD$, aims to bridge the gap between statistical efficiency and computational speed \citep{woodruff2014sketching}.  
Matrix $\bH_d\in\RR^{d \times d}$ is the Walsh-Hadamard matrix with $+1$ and $-1$ entries. $\bD \in\RR^{d \times d}$ is a diagonal matrix with diagonal entries sampled uniformly from $\{+1,-1\}$, and $bP\in\RR^{s\times d}$ is the uniform sampling matrix.  While the required $\ell$ is slightly higher than Gaussian matrices by logarithmic factors, the fast Walsh-Hadamard transform allows for $\bS\bA$ to be computed in $\mathcal{O}(d^2 \log \ell)$ (for $d \times d$ matrices), offering a significant speedup for dense data. 

\item \textbf{Sparse Embedding:} These matrices optimize for sparsity by maintaining only $s \ll \ell$ non-zero entries per column. Based on \citet{cohen2016simpler}, a sketch size of $\ell = \Omega(k \log(k/\delta) \gamma^{-2})$ and sparsity $s = \mathcal{O}(\log(k/\delta) \gamma^{-1})$ are sufficient for the $(\gamma, k, \delta)$-property. This construction is particularly effective for high-dimensional sparse data where preserving sparsity is crucial.
\end{itemize}

\subsection{Algorithm Description of Adaptive Zeroth-Order Method}

\begin{algorithm}[t]
	\caption{Adaptive Zeroth-Order Method for Deterministic Optimization }
	\label{alg:SA}
	\begin{small}
		\begin{algorithmic}[1]
			\STATE {\bf Input:}
			Initial vector $x_0$, sample size $\ell$, and step size $\eta$.
			\FOR {$t=0,1,2,\dots, T-1$ }
			\STATE Generate a $(\gamma, k, \delta)$-oblivious sketching matrix $\bS_t \in\RR^{d \times \ell}$ for the matrix product with $\si$ be the $i$-th column of $\bS_t$.
			\STATE Access to the value of $f(\cdot)$ and construct the approximate gradient
			\begin{equation*}
				\g(\x_t) = \sum_{i=1}^\ell \frac{f(\x_t+\alpha \si) - f(\x_t)}{\alpha} \si.
			\end{equation*}
			\STATE Compute the standard deviation as
			\begin{equation}\label{eq:sig}
				\sigma_t = \sqrt{
					\frac{1}{\ell - 1}\sum_{j=1}^{\ell} \left( f(\x_t + \alpha \bm{s}_j) - \frac{1}{\ell} \sum_{i=1}^{\ell} f(\x_t + \alpha \si) \right)^2}
			\end{equation}
			\STATE Update as
			\begin{equation}\label{eq:update}
				\x_{t+1} = \x_t - \eta\cdot \frac{\g(\x_t)}{\sigma_t} 
			\end{equation}
			\ENDFOR
			\STATE {\bf Output:} $\x_T$
		\end{algorithmic}
	\end{small}
\end{algorithm}

\begin{algorithm}[t]
	\caption{Adaptive Zeroth-Order Method for Deterministic Optimization }
	\label{alg:SA1}
	\begin{small}
		\begin{algorithmic}[1]
			\STATE {\bf Input:}
			Initial vector $x_0$, sample size $\ell$, and step size $\eta$.
			\FOR {$t=0,1,2,\dots, T-1$ }
			\STATE Generate a $(\gamma, k, \delta)$-oblivious sketching matrix $\bS_t \in\RR^{d \times \ell}$ for the matrix product with $\si$ be the $i$-th column of $\bS_t$.
			\STATE Access to the value of $f(\cdot)$ and construct the approximate gradient
			\begin{equation*}
				\hg(\x_t) = \sum_{i=1}^\ell \frac{f(\x_t+\alpha \si;\; \xi_t) - f(\x_t;\;\xi_t)}{\alpha} \si.
			\end{equation*}
			\STATE Compute the standard deviation as
			\begin{equation}\label{eq:sig1}
				\sigma_t = \sqrt{
					\frac{1}{\ell - 1}\sum_{j=1}^{\ell} \left( f(\x_t + \alpha \bm{s}_j;\;\xi_t) - \frac{1}{\ell} \sum_{i=1}^{\ell} f(\x_t + \alpha \si;\;\xi_t) \right)^2}
			\end{equation}
			\STATE Update as
			\begin{equation*}
				\x_{t+1} = \x_t - \eta\cdot\frac{\hg(\x_t)}{\sigma_t+\beta} 
			\end{equation*}
			\ENDFOR
			\STATE {\bf Output:} $\x_T$
		\end{algorithmic}
	\end{small}
\end{algorithm}

Let $\bS \in \mathbb{R}^{d \times \ell}$ be a pre-defined matrix, and  $\si$ denote its $i$-th column. 
We approximate the gradient of $f:\BR^d\to\BR$ at a given point $\x_t$ by
\begin{equation}
	\g(\x_t) = \sum_{i=1}^\ell \frac{f(\x_t+\alpha \si) - f(\x_t)}{\alpha} \cdot\si, \label{eq:g}
\end{equation}
were $\alpha>0$.
For stochastic minimization problem \eqref{eq:obj1}, we can only access noisy function evaluations $f(\,\cdot\,;\xi_t)$ to approximate the stochastic gradient $\nabla f(\x_t;\;\xi_t)$ as 
\begin{equation}
	\hg(\x_t) = \sum_{i=1}^\ell \frac{f(\x_t+\alpha \si;\xi_t) - f(\x_t;\xi_t)}{\alpha} \si. \label{eq:g1}
\end{equation} 

Furthermore, the standard deviations for the (noisy) function evaluations at point $\x_t+\alpha \si$ is
\begin{align*}
\sigma_t = \sqrt{
	\frac{1}{\ell - 1}\sum_{j=1}^{\ell} \left( f(\x_t + \alpha \bm{s}_j) - \frac{1}{\ell} \sum_{i=1}^{\ell} f(\x_t + \alpha \si) \right)^2}
\end{align*}
for the deterministic case
and
\begin{align*}
\sigma_t = \sqrt{
	\frac{1}{\ell - 1}\sum_{j=1}^{\ell} \left( f(\x_t + \alpha \bm{s}_j;\;\xi_t) - \frac{1}{\ell} \sum_{i=1}^{\ell} f(\x_t + \alpha \si;\;\xi_t) \right)^2}.
\end{align*}
for the stochastic cases.

Once the standard deviation obtained, we can conduct the update for the deterministic case as
\begin{align*}
\x_{t+1} = \x_t - \eta \cdot\frac{\g(\x_t)}{\sigma_t},
\end{align*}
where $\eta$ is the step size. 
Similarly, we conduct the update for the stochastic case as 
\begin{align*}
\x_{t+1} = \x_t - \eta\cdot \frac{\g(\x_t)}{\sigma_t + \beta},
\end{align*}
where $\beta$ is a positive constant.

The detailed algorithm descriptions of the adaptive zeroth-order method for deterministic and stochastic optimization are described in Algorithm~\ref{alg:SA} and Algorithm~\ref{alg:SA1}, respectively.

\section{Convergence Analysis for the Deterministic Problem}

This section present the convergence analysis of Algorithm~\ref{alg:SA} for the deterministic problem (\ref{eq:obj0}).

\subsection{Events}
Assume the objective function $f(\x)$ satisfies Assumption~\ref{ass:LL} and $\bS_t$ is an  oblivious $\left(\frac{1}{4}, k, \delta\right)$-random sketching matrix just as defined in Definition~\ref{def:ske} with $4 \leq k $ and $0 <\delta <1$. 
Then, we will define the following events
\begin{align}
	&\cE_{t,1} := \left\{ \sum_{i=1}^{\ell}  |\si^\top \bH \si|^2 \leq \zeta_1 \mid \si \mbox{ is the } i\mbox{-th column of } \bS_t \right\}, \label{eq:E1}\\
	&\cE_{t,2} := \left\{ \left(\frac{3}{4} - \frac{1}{k}\right) \norm{\nabla f(\x_t)}^2 
	\leq 	\nabla^\top f(\x_t) \bS_t\bS_t^\top \nabla f(\x_t) 
	\leq \left(\frac{5}{4} + \frac{1}{k}\right)  \norm{\nabla f(\x_t)}^2 \right\}, \label{eq:E2}\\
	&\cE_{t,3} := \left\{\norm{ \bS_t^\top \bH \bS_t }_2 \leq \zeta_2\right\}, \label{eq:E3}\\
	&\cE_{t,4} := \left\{\frac{1}{\ell}  \left( \sum_{i=1}^{\ell} \dotprod{\nabla f(\x_t), \si} \right)^2 \leq \frac{1}{4}\norm{\nabla f(\x_t)}^2\right\}, \label{eq:E4}
\end{align} 
where $\zeta_1$, and $\zeta_2$ are defined in Eq.~\eqref{eq:zeta_1}, and Eq.~\eqref{eq:zeta_2}, respectively.

Next, we will bound the probabilities of above events happen.
\begin{lemma}\label{lem:E1}
Letting $\bS_t$ be an  oblivious $\left({1}/{4}, k, \delta\right)$-random sketching matrix follows Definition~\ref{def:ske} with $k\geq 4 $ and $\delta\in(0,1)$, then event $\cE_{t,1}$ holds with a probability at least $1-\delta$ with $\zeta_1$ defined as	
\begin{equation}\label{eq:zeta_1}
	\zeta_1 := \norm{\bH}^2 \left(\frac{5\ell}{4} + \frac{d \ell}{4k}\right)^2.
\end{equation}
\end{lemma}
\begin{proof}
	We have
\begin{align*}
\sum_{i=1}^{\ell}  |\si^\top \bH \si|^2 
\leq
\norm{\bH}^2 \sum_{i=1}^{\ell} \norm{\si}^4
\leq 
\norm{\bH}^2 \norm{\bS_t}_F^4
\stackrel{\eqref{eq:S_norm}}{\leq}
\norm{\bH}^2 \left(\frac{5\ell}{4} + \frac{d \ell}{4k}\right)^2,
\end{align*}
where the last inequality is because of Lemma~\ref{lem:bound_S_F_norm}.
\end{proof}

\begin{lemma}\label{lem:E2}
	Letting $\bS\in\RR^{\ell\times d}$ is an oblivious $\left({1}/{4}, k, \delta\right)$-random sketching matrix follows Definition~\ref{def:ske} with $k\geq 4$ and $\delta\in(0,1)$, then event $\cE_{t,2}$ holds with  a probability at least $1-\delta$.
\end{lemma}
\begin{proof}
	By Eq.~\eqref{eq:mat_pd}, we can obtain that
	\begin{align*}
		\norm{ \nabla^\top f(\x_t) \bS \bS^\top \nabla f(\x_t) -  \nabla^\top f(\x_t) \nabla f(\x_t)} 
		\leq 
		\frac{1}{4}  \left( \norm{\nabla f(\x_t)}^2 + \frac{\norm{\nabla f(\x_t)}^2}{k} \right).
	\end{align*}
	Furthermore, it holds that
	\begin{align*}
		\norm{ \nabla^\top f(\x_t) \bS \bS^\top \nabla f(\x_t)  } - \norm{\nabla f(\x_t)}^2 
		\leq&  
		\norm{ \nabla^\top f(\x_t) \bS \bS^\top \nabla f(\x_t) -  \nabla^\top f(\x_t) \nabla f(\x_t)},  \\
		\norm{\nabla f(\x_t)}^2 - \norm{ \nabla^\top f(\x_t) \bS \bS^\top \nabla f(\x_t)  }  
		\leq& 
		\norm{ \nabla^\top f(\x_t) \bS \bS^\top \nabla f(\x_t) -  \nabla^\top f(\x_t) \nabla f(\x_t)}. 
	\end{align*}
	Thus, we can obtain that
	\begin{align*}
		\left(\frac{3}{4} - \frac{1}{k}\right) \norm{\nabla f(\x_t)}^2 \leq 	\norm{ \nabla^\top f(\x_t) \bS \bS^\top \nabla f(\x_t)  } \leq \left(\frac{5}{4} + \frac{1}{k}\right)  \norm{\nabla f(\x_t)}^2 .
	\end{align*}
\end{proof}

\begin{lemma}\label{lem:E3}
	Letting $\bS\in\RR^{d\times \ell}$ be  an oblivious $({1}/{4}, k, \delta)$-random sketching matrix, then event $\cE_{t,3}$ holds with a probability at least $1-\delta$ with $\zeta_2$ defined as 
	\begin{equation}\label{eq:zeta_2}
		\zeta_2 := \frac{5\norm{\bH}}{4} +  \frac{\tr(\bH)}{4k}.
	\end{equation} 
\end{lemma}
\begin{proof}
	First, we have
	\begin{align*}
		\norm{ \bS^\top \bH \bS } 
		=& 
		\norm{ \bH^{1/2} \bS\bS^\top\bH^{1/2}}.
	\end{align*}	
	By the property that  $\bS$ is an oblivious $({1}/{4}, k, \delta)$-random sketching matrix, then we can obtain that
	\begin{align*}
		\norm{ \bH^{1/2} \bS\bS^\top \bH^{1/2} -  \bH} \leq \frac{1}{4} \left( \norm{\bH}_2 + \frac{\tr( \bH )}{k} \right),
	\end{align*}
	which implies that
	\begin{align*}
		\norm{ \bH^{1/2} \bS\bS^\top \bH^{1/2} } 
		\leq& 
		\frac{5}{4} \norm{\bH }_2 + \frac{\tr( \bH )}{4k}.
	\end{align*}
\end{proof}

\begin{lemma}\label{lem:E4}
	If $\bS\in\RR^{d\times \ell}$ is random sketching matrix which is Gaussian, Rademacher, SRHT, or Sparse Embedding (refer to Sec.~\ref{subsec:Sketch}), then event $\cE_{t,4}$ holds  with  a probability at least $1 - \delta'$ with $\delta' = \exp\left(-\cO(\ell)\right)$.
\end{lemma}
\begin{proof}
Refer to Lemma~\ref{lem:ss}.
\end{proof}

\subsection{Properties of Estimated Gradient and Standard Deviation}

First, we will provide several important properties related to the estimated gradient defined in Eq.~\eqref{eq:g}.
\begin{lemma}
	Under Assumptions \ref{ass:LL} and \ref{ass:lower-bounded}, 
	we additionally suppose the gradient $\g(\x)$ defined in Eq.~\eqref{eq:g} has the form of 
	\begin{equation}\label{eq:g_prop}
		\g(\x) = \bS\bS^\top \nabla f(\x) + \frac{\alpha}{4}\bS\bv \quad\mbox{ and }\quad |v^{(i)}| \le  2\left(L_0 + L_1 \norm{\nabla f(\x)}\right) \norm{\si}_{\bH}^2, 
	\end{equation} 
    where $\bS=[\s_1,\dots,\s_\ell]\in\RR^{d \times \ell}$ is a given matrix and $\bv=[v^{(i)},\dots,v^{(\ell)}]\in\BR^{\ell}$ is an $\ell$-dimension vector with $v^{(i)} = 4\alpha^{-1}\gamma(\x, \alpha \si)$ 
    and 
	\begin{equation}\label{eq:gamma}
		\gamma(\x, \alpha \si) = f(\x+\alpha \si) - f(\x) - \alpha \dotprod{\nabla f(\x), \si}.
	\end{equation}
\end{lemma}
\begin{proof}
We have
\begin{align*}
	\sum_{i=1}^{\ell}\frac{f(\x+\alpha \si) - f(\x)}{\alpha}\si 
	\stackrel{\eqref{eq:gamma}}{=}& 
	\sum_{i=1}^{\ell} \frac{\alpha \dotprod{\nabla f(\x), \si} + \gamma(\x, \alpha \si) }{\alpha} \si\\
	=&
	\sum_{i=1}^{\ell} \dotprod{\nabla f(\x), \si} \si 
	+  \alpha^{-1} \sum_{i=1}^{\ell} \gamma(\x, \alpha \si) \cdot \si \\
	=& \bS\bS^\top \nabla f(\x) +  \frac{\alpha}{4}\bS \bv,
\end{align*}
where $v^{(i)} = 4\alpha^{-1} \cdot \gamma(\x, \alpha \si) $.

Furthermore, by Eq.~\eqref{eq:LL}, we have
\begin{align*}
\left|
4 \alpha^{-1} \cdot \gamma(\x, \alpha \si)
\right|
\stackrel{\eqref{eq:LL}}{\leq}
\frac{4}{\alpha} \cdot \frac{L_0 + L_1 \norm{\nabla f(\x)}}{2} \cdot \alpha^2 \norm{\si}_{\bH}^2
=
2\left(L_0 + L_1 \norm{\nabla f(\x)}\right) \norm{\si}_{\bH}^2.
\end{align*}

\end{proof}

\begin{lemma}
Let $\bv=[v^{(i)},\dots,v^{(\ell)}]\in\BR^{\ell}$ be an $\ell$-dimension vector with $v^{(i)} = 4\alpha^{-1} \cdot \gamma(\x, \alpha \si)$, where $\gamma(\x, \alpha \si)$ following the definition in Eq.~\eqref{eq:gamma}. 
Conditioned on event $\cE_{t,1}$, the quantity $\norm{\bv}^2$ is upper bounded by	
	\begin{equation}\label{eq:v_u}
		\norm{\bv}^2
		\leq
		4(L_0 + L_1 \norm{\nabla f(\x)})^2 \zeta_1,
	\end{equation}
where $\zeta_1$ is defined in Eq.~\eqref{eq:zeta_1}.
\end{lemma}
\begin{proof}
By the definition of $\bv$, we can obtain that
	\begin{align*}
		\norm{\bv}^2 
		=
		\sum_{i=1}^{\ell} \left(v^{(i)}\right)^2
		\stackrel{\eqref{eq:g_prop}}{\leq}
		4(L_0 + L_1 \norm{\nabla f(\x)})^2  \sum_{i=1}^{\ell} \norm{\si}_{\bH}^4
		\stackrel{\eqref{eq:E1}}{\leq}
		4(L_0 + L_1 \norm{\nabla f(\x)})^2 \zeta_1,
	\end{align*}
	which concludes the proof.
\end{proof}

\begin{lemma}\label{lem:gamma}
	Letting $\gamma(\x, \alpha \si)$  defined in Eq.~\eqref{eq:gamma}, conditioned on event $\cE_{t,1}$ defined in Eq.~\eqref{eq:E1}, then it holds that
	\begin{equation}\label{eq:gamma_up}
		\sum_{i=1}^{\ell} \gamma(\x, \alpha \si)^2 
		\leq 
		\frac{\alpha^4(L_0^2 + L_1^2 \norm{\nabla f(\x)}^2)}{4} \cdot \zeta_1.
	\end{equation}
\end{lemma}
\begin{proof}
	By the definition of $\gamma$ in Eq.~\eqref{eq:gamma} and  the assumption in Eq.~\eqref{eq:LL}, we can obtain that
	\begin{align*}
		|\gamma(\x, \alpha \si) | \leq \alpha^2 \cdot \frac{L_0 + L_1 \norm{\nabla f(\x)}}{2} |\si^\top \bH \si|.
	\end{align*}
	Therefore, we can obtain that
	\begin{align*}
		\sum_{i=1}^{\ell} \gamma(\x, \alpha \si)^2 
		\leq
		\alpha^4 \cdot \frac{(L_0 + L_1 \norm{\nabla f(\x)})^2}{4} \cdot \sum_{i=1}^{\ell}  |\si^\top \bH \si|^2
		\stackrel{\eqref{eq:E1}}{\leq}
		\frac{\alpha^4(L_0 + L_1 \norm{\nabla f(\x)})^2  }{4} \cdot \zeta_1.
	\end{align*}
\end{proof}

Next, we will prove properties of the standard deviation $\sigma$ defined in Eq.~\eqref{eq:sig}.
\begin{lemma}\label{lem:sig}
Letting the standard deviation $\sigma$ defined in Eq.~\eqref{eq:sig}, then it holds
\begin{equation}\label{eq:sig_up0}
\sigma^2 
\leq 
\frac{3\alpha^2}{2(\ell-1)} \sum_{i=1}^{\ell} \dotprod{ \nabla f(\x), \bm{s}_i }^2 
+ \frac{3}{\ell - 1} 	\sum_{i=1}^{\ell}  \gamma(\x, \alpha \bm{s}_j)^2
\end{equation}
and
\begin{equation}\label{eq:sig_low0}
\sigma^2 
\geq 
\frac{\alpha^2}{2(\ell-1)} \sum_{i=1}^{\ell} \left(\dotprod{ \nabla f(\x), \bm{s}_i }^2  - \frac{1}{\ell^2}  \left( \sum_{i=1}^{\ell} \dotprod{\nabla f(\x), \si} \right)^2 \right) - \frac{1}{\ell - 1} 	\sum_{i=1}^{\ell}  \gamma(\x, \alpha \bm{s}_j)^2,
\end{equation}
where $\gamma(\x, \alpha \si)$ is defined in Eq.~\eqref{eq:gamma}.
\end{lemma}
\begin{proof}
By the definition of $\gamma(\cdot, \cdot)$ in Eq.~\eqref{eq:gamma}, we have 
\begin{align*}
	&f(\x + \alpha \bm{s}_j) - \frac{1}{\ell} \sum_{i=1}^{\ell} f(\x + \alpha \si) \\
	=&
	f(\x) + \alpha \dotprod{\nabla f(\x), \bm{s}_j} + \gamma(\x, \alpha \bm{s}_j) - \frac{1}{\ell} \sum_{i=1}^{\ell}\left( f(\x) + \alpha \dotprod{\nabla f(\x), \si} + \gamma(\x, \alpha \si) \right)\\
	=&
	\alpha \dotprod{\nabla f(\x), \bm{s}_j - \frac{1}{\ell} \sum_{i=1}^{\ell}\si} + \gamma(\x, \alpha \bm{s}_j) - \frac{1}{\ell}\sum_{i=1}^{\ell} \gamma(\x, \alpha \si).
\end{align*}
Defining $\tau_j = \gamma(\x, \alpha \bm{s}_j) - \frac{1}{\ell}\sum_{i=1}^{\ell} \gamma(\x, \alpha \si)$, we can obtain that
\begin{equation*}
\begin{aligned}
	&\sum_{j=1}^{\ell} \left( f(\x + \alpha \bm{s}_j) - \frac{1}{\ell} \sum_{i=1}^{\ell} f(\x + \alpha \si) \right)^2\\
	=&
	\sum_{j=1}^{\ell} \left( \alpha \dotprod{ \nabla f(\x), \bm{s}_j - \frac{1}{\ell} \sum_{i=1}^{\ell}\si } + \tau_j \right)^2\\
	=&
	\sum_{j=1}^{\ell} \left( \alpha^2 \dotprod{ \nabla f(\x), \bm{s}_j - \frac{1}{\ell} \sum_{i=1}^{\ell}\si }^2 + \tau_j^2 + 2\alpha \dotprod{ \nabla f(\x), \bm{s}_j - \frac{1}{\ell} \sum_{i=1}^{\ell}\si } \tau_j \right).
\end{aligned}
\end{equation*}

Furthermore,
\begin{align*}
2\alpha \left| \dotprod{ \nabla f(\x), \bm{s}_j - \frac{1}{\ell} \sum_{i=1}^{\ell}\si } \tau_j \right| 
\leq
\frac{\alpha^2}{2}   \dotprod{ \nabla f(\x), \bm{s}_j - \frac{1}{\ell} \sum_{i=1}^{\ell}\si }^2 + 2 \tau_j^2.
\end{align*}
Therefore, it holds that
\begin{equation}\label{eq:ll}
	\sum_{j=1}^{\ell} \left( f(\x + \alpha \bm{s}_j) - \frac{1}{\ell} \sum_{i=1}^{\ell} f(\x + \alpha \si) \right)^2 
	\leq \sum_{j=1}^{\ell} \left( \frac{3\alpha^2}{2} \dotprod{ \nabla f(\x), \bm{s}_j - \frac{1}{\ell} \sum_{i=1}^{\ell}\si }^2 + 3\tau_j^2\right),
\end{equation}
and
\begin{equation}\label{eq:gg}
\sum_{j=1}^{\ell} \left( \frac{\alpha^2}{2} \dotprod{ \nabla f(\x), \bm{s}_j - \frac{1}{\ell} \sum_{i=1}^{\ell}\si }^2 -\tau_j^2\right) \leq	\sum_{j=1}^{\ell} \left( f(\x + \alpha \bm{s}_j) - \frac{1}{\ell} \sum_{i=1}^{\ell} f(\x + \alpha \si) \right)^2.
\end{equation}

Next, we will bound the value of $\dotprod{ \nabla f(\x), \bm{s}_j - \frac{1}{\ell} \sum_{i=1}^{\ell}\si }^2$ and $\tau_j^2$ as follows
\begin{equation}\label{eq:exp}
\begin{aligned}
	&\sum_{j=1}^{\ell} \dotprod{ \nabla f(\x), \bm{s}_j - \frac{1}{\ell} \sum_{i=1}^{\ell}\si }^2 \\
	=& 
	\sum_{j=1}^{\ell} \left( \dotprod{\nabla f(\x), \bm{s}_j}^2 -2 \dotprod{\nabla f(\x), \bm{s}_j} \cdot \left(\frac{1}{\ell}\sum_{i=1}^{\ell} \dotprod{\nabla f(\x), \si}\right) + \left( \frac{1}{\ell}\sum_{i=1}^{\ell} \dotprod{\nabla f(\x), \si} \right)^2 \right)\\
	=&
	\sum_{i=1}^{\ell} \dotprod{\nabla f(\x), \si}^2 - \frac{2}{\ell} \left( \sum_{i=1}^{\ell} \dotprod{\nabla f(\x), \si} \right)^2 + \frac{1}{\ell} \left( \sum_{i=1}^{\ell} \dotprod{\nabla f(\x), \si} \right)^2\\
	=&
	\sum_{i=1}^{\ell} \dotprod{\nabla f(\x), \si}^2 - \frac{1}{\ell}  \left( \sum_{i=1}^{\ell} \dotprod{\nabla f(\x), \si} \right)^2.
\end{aligned}
\end{equation}

Furthermore, it holds that
\begin{equation}\label{eq:tau}
	0 \leq \sum_{j=1}^{\ell} \tau_j^2 = \sum_{j = 1}^{\ell} \left( \gamma(\x, \alpha \bm{s}_j) - \frac{1}{\ell}\sum_{i=1}^{\ell} \gamma(\x, \alpha \si) \right)^2
	\leq 
	\sum_{i=1}^{\ell}  \gamma(\x, \alpha \bm{s}_j)^2.
\end{equation}

Combining above results, we can obtain that
\begin{align*}
	\sigma^2 
	\stackrel{\eqref{eq:sig}}{=}&
	\frac{1}{\ell - 1}\sum_{j=1}^{\ell} \left( f(\x + \alpha \bm{s}_j) - \frac{1}{\ell} \sum_{i=1}^{\ell} f(\x + \alpha \si) \right)^2\\
	\stackrel{\eqref{eq:ll}}{\leq}&
	\frac{1}{\ell-1} \sum_{j=1}^{\ell} \left( \frac{3\alpha^2}{2} \dotprod{ \nabla f(\x), \bm{s}_j - \frac{1}{\ell} \sum_{i=1}^{\ell}\si }^2 + 3\tau_j^2\right)\\
	\stackrel{\eqref{eq:exp}}{=}&
	\frac{3\alpha^2}{2(\ell-1)} \sum_{i=1}^{\ell} \left(\dotprod{ \nabla f(\x), \bm{s}_i }^2  - \frac{1}{\ell^2}  \left( \sum_{i=1}^{\ell} \dotprod{\nabla f(\x), \si} \right)^2 \right) 
	+ \frac{3}{\ell-1} \sum_{i=1}^{\ell} \tau_i^2\\
	\stackrel{\eqref{eq:tau}}{\leq}&
	\frac{3\alpha^2}{2(\ell-1)} \sum_{i=1}^{\ell} \left(\dotprod{ \nabla f(\x), \bm{s}_i }^2  - \frac{1}{\ell^2} \left( \sum_{i=1}^{\ell} \dotprod{\nabla f(\x), \si} \right)^2 \right) 
	+ \frac{3}{\ell - 1} 	\sum_{i=1}^{\ell}  \gamma(\x, \alpha \bm{s}_j)^2\\
	\leq&
	\frac{3\alpha^2}{2(\ell-1)} \sum_{i=1}^{\ell} \dotprod{ \nabla f(\x), \bm{s}_i }^2 
	+ \frac{3}{\ell - 1} 	\sum_{i=1}^{\ell}  \gamma(\x, \alpha \bm{s}_j)^2,
\end{align*}
and
\begin{align*}
	\sigma^2 =&
\frac{1}{\ell - 1}\sum_{j=1}^{\ell} \left( f(\x + \alpha \bm{s}_j) - \frac{1}{\ell} \sum_{i=1}^{\ell} f(\x + \alpha \si) \right)^2\\
\stackrel{\eqref{eq:gg}}{\geq}&
\frac{1}{\ell-1} \sum_{j=1}^{\ell} \left( \frac{\alpha^2}{2} \dotprod{ \nabla f(\x), \bm{s}_j - \frac{1}{\ell} \sum_{i=1}^{\ell}\si }^2 - \tau_j^2\right)\\
\stackrel{\eqref{eq:exp}}{=}&
\frac{\alpha^2}{2(\ell-1)} \sum_{i=1}^{\ell} \left(\dotprod{ \nabla f(\x), \bm{s}_i }^2  - \frac{1}{\ell^2}  \left( \sum_{i=1}^{\ell} \dotprod{\nabla f(\x), \si} \right)^2 \right) 
- \frac{1}{\ell-1} \sum_{i=1}^{\ell} \tau_i^2\\
\stackrel{\eqref{eq:tau}}{\geq}&
\frac{\alpha^2}{2(\ell-1)} \sum_{i=1}^{\ell} \left(\dotprod{ \nabla f(\x), \bm{s}_i }^2  - \frac{1}{\ell^2}  \left( \sum_{i=1}^{\ell} \dotprod{\nabla f(\x), \si} \right)^2 \right) - \frac{1}{\ell - 1} 	\sum_{i=1}^{\ell}  \gamma(\x, \alpha \bm{s}_j)^2.
\end{align*}

\end{proof}

\begin{lemma}
	Assume that the smooth parameter of Algorithm~\ref{alg:SA} satisfies  
    \begin{align*}
    \alpha^2 \leq \min\left\{ \frac{1}{2L_1^2\zeta_1},\; \frac{3\norm{\nabla f(\x_t)}^2}{8L_0^2\zeta_1}\right\}
    \end{align*}
    with $\zeta_1$ defined in Eq.~\eqref{eq:zeta_1}. 
	Conditioned on events $\cE_{t,2}$ and $\cE_{t,4}$ defined in Eq.~\eqref{eq:E2} and Eq.~\eqref{eq:E4} respectively, then  it holds that
	\begin{equation}\label{eq:sig_low}
		\sigma_t \geq \frac{\alpha}{8\sqrt{(\ell-1)}} \norm{\nabla f(\x_t)}.
	\end{equation}
\end{lemma}
\begin{proof}
	We have 
	\begin{align*}
		\sigma_t^2 
		\stackrel{\eqref{eq:sig_low0}}{\geq}& 
		\frac{\alpha^2}{2(\ell-1) } \sum_{i=1}^\ell \left(  \dotprod{\nabla f(\x_t), \si}^2 - \frac{1}{4\ell} \norm{\nabla f(\x_t)}^2 \right) - \frac{1}{\ell - 1} 	\sum_{i=1}^{\ell}  \gamma(\x, \alpha \bm{s}_j)^2\\
		\geq&
		\frac{\alpha^2}{2(\ell-1)} \cdot \frac{\norm{\nabla f(\x_t)}^2}{4} - \frac{1}{\ell - 1} 	\sum_{i=1}^{\ell}  \gamma(\x, \alpha \bm{s}_j)^2,
	\end{align*}
	where the first inequality is because of the condition $  \frac{1}{\ell}  \left( \sum_{i=1}^{\ell} \dotprod{\nabla f(\x), \si} \right)^2 \leq \frac{1}{4}\norm{\nabla f(\x)}^2$ and the last inequality is because of Lemma~\ref{lem:E2}.
	
	By Lemma~\ref{lem:gamma}, we can obtain that
	\begin{align*}
		\sigma_t^2 
		\stackrel{\eqref{eq:gamma_up}}{\geq}&
		\frac{\alpha^2}{8(\ell - 1)}  \cdot \norm{\nabla f(\x_t)}^2 - \frac{\alpha^4 (L_0^2 + L_1^2 \norm{\nabla f(\x_t)}^2 ) \zeta_1}{4(\ell-1)}\\
		=&
		\frac{\alpha^2}{8(\ell - 1)} \left( 1 - \alpha^2 L_1^2 \zeta_1 \right) \norm{\nabla f(\x_t)}^2   - \frac{\alpha^4 L_0^2 \zeta_1}{4(\ell - 1)}\\
		\geq&
		\frac{\alpha^2}{16(\ell-1)} \norm{\nabla f(\x_t)}^2 - \frac{\alpha^4 L_0^2 \zeta_1}{4(\ell - 1)}\\
		\geq&
		\frac{\alpha^2}{64(\ell-1)} \norm{\nabla f(\x_t)}^2,
	\end{align*}
	where the last two inequalities are because of the assumption on $\alpha^2$.	
\end{proof}

\begin{lemma}
	Assume that the smooth parameter of Algorithm~\ref{alg:SA} satisfies  
    \begin{align*}
    \alpha^2 \leq \min\{ \frac{1}{2L_1^2\zeta_1},\; \frac{3\norm{\nabla f(\x_t)}^2}{8L_0^2\zeta_1}\}.    
    \end{align*}
    Conditioned on event $\cE_{t,2}$ defined in Eq.~\eqref{eq:E2},  it holds that
	\begin{equation}\label{eq:sig_up}
		\sigma_t \leq \frac{13 \norm{\nabla f(\x_t)}}{2 \sqrt{\ell - 1}}.
	\end{equation}
\end{lemma}
\begin{proof}
	By Lemma~\ref{lem:sig}, we have
	\begin{align*}
		\sigma_t^2 
		\stackrel{\eqref{eq:sig_up0}}{\leq}& 
		\frac{3\alpha^2}{2(\ell-1)} \sum_{i=1}^{\ell} \dotprod{ \nabla f(\x_t), \bm{s}_i }^2 
		+ \frac{3}{\ell - 1} 	\sum_{i=1}^{\ell}  \gamma(\x_t, \alpha \bm{s}_j)^2\\
		\stackrel{\eqref{eq:E2}}{\leq}&
		\frac{9\alpha^2}{4(\ell-1)}\norm{\nabla f(\x_t)}^2
		+ \frac{3}{\ell - 1} 	\sum_{i=1}^{\ell}  \gamma(\x_t, \alpha \bm{s}_j)^2\\
		\leq&
		\frac{9\alpha^2}{4(\ell-1)}\norm{\nabla f(\x_t)}^2
		+ \frac{3\alpha^4 (L_0^2 + L_1^2 \norm{\nabla f(\x_t)}^2)\zeta_1}{4(\ell - 1)}, 
	\end{align*}
	where the last inequality is because of Lemma~\ref{lem:gamma}.
	
	Combining with the condition on $\alpha^2$, we can obtain that
	\begin{align*}
		\sigma_t^2
		\leq 
		\left(\frac{9}{4} + \frac{3}{4}\right) \frac{\alpha^2\norm{\nabla f(\x_t)}^2}{\ell - 1} 
		+ \frac{3\alpha^4L_0^2\zeta_1}{4(\ell-1)}
		\leq
		\left(\frac{9}{4} + \frac{3}{4} + \frac{9}{64}\right) \frac{\alpha^2\norm{\nabla f(\x_t)}^2}{\ell - 1},
	\end{align*}
	which implies that
	\begin{align*}
		\sigma_t \leq \frac{13 \alpha \norm{\nabla f(\x_t)}}{2 \sqrt{\ell - 1}}.
	\end{align*}
\end{proof}

\begin{lemma}
Letting events $\cE_{t,1}$,  $\cE_{t,2}$,  $\cE_{t,3}$, and $\cE_{t,4}$ hold and the smooth parameter of Algorithm~\ref{alg:SA} satisfies 
\begin{align*}
\alpha^2 \leq \min\{ \frac{1}{2L_1^2\zeta_1},\; \frac{3\norm{\nabla f(\x_t)}^2}{8L_0^2\zeta_1}\},    
\end{align*}
then it holds that
\begin{equation}\label{eq:length}
\frac{\norm{\g(\x_t)}_{\bH}}{\sigma_t}
\leq 
\frac{12\sqrt{\zeta_2}\sqrt{\ell - 1}}{\alpha} + \frac{ 4\sqrt{2} \sqrt{\zeta_2}  }{\alpha}.
\end{equation}
\end{lemma}
\begin{proof}
	We have
\begin{align*}
\frac{\norm{\g(\x_t)}_{\bH}}{\sigma_t}
\stackrel{\eqref{eq:g_prop}}{\leq}&
\frac{\norm{\bS\bS^\top \nabla f(\x_t)}_{\bH} + \frac{\alpha}{4}\norm{\bS \bv}_{\bH}}{\sigma_t}
\leq
\frac{\sqrt{\zeta_2} \left( \norm{\bS^\top \nabla f(\x_t)} + \frac{\alpha}{4} \norm{\bv}\right)}{\sigma_t}\\
\leq&
\frac{\sqrt{\zeta_2}\left( \frac{3}{2}\norm{\nabla f(\x_t)} + \frac{\alpha}{4} \sqrt{4(L_0 + L_1 \norm{\nabla f(\x)})^2 \zeta_1}  \right)}{\sigma_t}\\
\stackrel{\eqref{eq:sig_low}}{\leq}&
\frac{\sqrt{\zeta_2}\left( \frac{3}{2}\norm{\nabla f(\x_t)} + \frac{\alpha}{4} \sqrt{4(L_0 + L_1 \norm{\nabla f(\x)})^2 \zeta_1}  \right)}{\frac{\alpha}{8\sqrt{(\ell-1)}} \norm{\nabla f(\x_t)}}\\
=&
\frac{12\sqrt{\zeta_2}\sqrt{\ell - 1}}{\alpha}
+ 
\frac{4 \sqrt{\zeta_2} \sqrt{\zeta_1} (L_0 + L_1 \norm{\nabla f(\x)})}{\norm{\nabla f(\x_t)}}\\
=&
\frac{12\sqrt{\zeta_2}\sqrt{\ell - 1}}{\alpha} + \frac{ 4 \sqrt{\zeta_2} \sqrt{\zeta_1} L_1 \alpha }{\alpha} + \frac{4 \sqrt{\zeta_2} \sqrt{\zeta_1} L_0 \alpha }{\alpha \norm{\nabla f(\x_t)}}\\
\leq&
\frac{12\sqrt{\zeta_2}\sqrt{\ell - 1}}{\alpha} + \frac{ 2\sqrt{2} \sqrt{\zeta_2}  }{\alpha} + \frac{2\sqrt{2}\sqrt{\zeta_2} }{\alpha}\\
=&
\frac{12\sqrt{\zeta_2}\sqrt{\ell - 1}}{\alpha} + \frac{ 4\sqrt{2} \sqrt{\zeta_2}  }{\alpha},
\end{align*}
where the last inequality is because of the condition on $\alpha^2$.
\end{proof}

\subsection{Descent Analysis and Main Results}
\begin{lemma}
Suppose the function $f:\BR^d\to\BR$ satisfy Assumption~\ref{ass:LL} and event $\cE_{t,3}$ defined in Eq.~\eqref{eq:E3} hold. 
After Algorithm~\ref{alg:SA} takes one update with step size $\eta \leq {3\alpha}/{(64 L_1 \zeta_2 \sqrt{\ell - 1})}$, it holds 
\begin{equation}\label{eq:ff_dec}
\begin{aligned}
f(\x_{t+1})
\leq &
f(\x_t) - \frac{7\eta}{8\sigma_t} \norm{ \bS^\top \nabla f(\x_t) }^2 + \frac{\eta\alpha^2\norm{\bv_t}^2}{8\sigma_t} \\
&+ \frac{\eta^2}{\sigma_t^2} \cdot \frac{L_0 + L_1 \norm{\nabla f(\x_t)}}{2} \cdot \left( 2\zeta_2 \cdot \norm{\bS^\top \nabla f(\x_t)}^2 +  \frac{\alpha^2\zeta_2}{8}  \norm{\bv_t}^2 \right),
\end{aligned}
\end{equation}
where $\zeta_2$ is defined in Eq.~\eqref{eq:zeta_2}.
\end{lemma}
\begin{proof}
First, by the update rule,  we have
\begin{align*}
\norm{\x_{t+1} - \x_t}_H 
=& 
\eta \cdot \frac{\norm{\g(\x_t)}_{\bH}}{\sigma_t}
\stackrel{\eqref{eq:length}}{\leq}
\eta \left( \frac{12\sqrt{\zeta_2}\sqrt{\ell - 1}}{\alpha} + \frac{ 4\sqrt{2} \sqrt{\zeta_2}  }{\alpha} \right)
\leq 
\frac{9}{16L_1 \sqrt{\zeta_2}} + \frac{3\sqrt{2}}{16 L_1 \sqrt{\ell - 1} \sqrt{\zeta_2}}\\
\leq& 
\frac{1}{L_1 \sqrt{\zeta_2}}
\leq 
\frac{1}{L_1 \sqrt{\norm{\bH}}},
\end{align*}
where the last inequality is because of definition of $\zeta_2$ in Eq.~\eqref{eq:zeta_2}.
Thus, the conditions in Assumption~\ref{ass:LL} hold.
Then, by the update step in Eq.~\eqref{eq:update} and $(L_0, L_1)$-smooth of $f(\x)$ in Assumption~\ref{ass:LL}, we can obtain that
\begin{align*}
	f(\x_{t+1}) 
	\stackrel{\eqref{eq:LL}}{\leq}&
	f(\x_t) + \dotprod{\nabla f(\x_t), \x_{t+1} - \x_t} + \frac{L_0 + L_1 \norm{\nabla f(\x_t)}}{2} ( \x_{t+1} - \x_t)^\top \bH  (\x_{t+1} - \x_t)\\
	=&
	f(\x_t) - \frac{\eta}{\sigma_t} \dotprod{\nabla f(\x_t), \g(\x_t)} + \frac{\eta^2}{\sigma_t^2} \frac{L_0 + L_1 \norm{\nabla f(\x_t)}}{2} \g(\x_t)^\top \bH \g(\x_t)\\
	\stackrel{\eqref{eq:g_prop}}{=}&
	f(\x_t) - \frac{\eta}{\sigma_t} \dotprod{\nabla f(\x_t), \bS\bS^\top \nabla f(\x_t) + \frac{\alpha}{4}\bS \bv_t} \\
	&+ \frac{\eta^2}{\sigma_t^2} \frac{L_0 + L_1 \norm{\nabla f(\x_t)}}{2} \left( \bS\bS^\top \nabla f(\x_t) + \frac{\alpha}{4}S \bv_t \right)^\top \bH \left( \bS\bS^\top \nabla f(\x_t) + \frac{\alpha}{4}S \bv_t \right).
\end{align*}
Furthermore, it holds that
\begin{align*}
& - \dotprod{\nabla f(\x_t), \bS\bS^\top \nabla f(\x_t) + \frac{\alpha}{4}\bS \bv_t}
=
-\norm{ \bS^\top \nabla f(\x_t) }^2 - \frac{1}{4}\dotprod{\bS^\top\nabla f(\x_t), \alpha \bv_t }\\
\leq&
-\norm{ \bS^\top \nabla f(\x_t) }^2 + \frac{1}{8} \norm{ \bS^\top\nabla f(\x_t) }^2 + \frac{\alpha^2 \norm{\bv_t}^2}{8}\\
=&
-\frac{7}{8}\norm{ \bS^\top \nabla f(\x_t) }^2 + \frac{\alpha^2 \norm{\bv_t}^2}{8}.
\end{align*}
We also have
\begin{align*}
&\left( \bS\bS^\top \nabla f(\x_t) + \frac{\alpha}{4}\bS \bv_t \right)^\top \bH \left( \bS\bS^\top \nabla f(\x_t) + \frac{\alpha}{4}\bS \bv_t \right)\\
=&
\nabla^\top f(\x_t) \bS\bS^\top \bH \bS\bS^\top \nabla f(\x_t) + \frac{\alpha}{2} \bv_t^\top \bS^\top \bH \bS \bS^\top \nabla f(\x_t) + \frac{\alpha^2}{16} \bv_t^\top \bS^\top \bH \bS \bv_t\\
\leq&
\norm{ \bS^\top \bH \bS}_2 \cdot \norm{\bS^\top \nabla f(\x_t)}^2 
+ \frac{\alpha}{2} \norm{\bv_t} \norm{ \bS^\top \bH \bS}_2 \norm{ \bS^\top \nabla f(\x_t)  } 
+ \frac{\alpha^2}{16} \norm{ \bS^\top \bH \bS }_2 \norm{\bv_t}^2\\
\leq&
\norm{ \bS^\top \bH \bS}_2 \cdot \norm{\bS^\top \nabla f(\x_t)}^2  + \left( \norm{ \bS^\top \bH \bS}_2 \cdot \norm{\bS^\top \nabla f(\x_t)}^2 + \frac{\alpha^2}{16} \norm{ \bS^\top \bH \bS }_2 \norm{\bv_t}^2\right) \\
&+ \frac{\alpha^2}{16} \norm{ \bS^\top \bH \bS }_2 \norm{\bv_t}^2\\
=&
2 \norm{ \bS^\top \bH \bS}_2 \cdot \norm{\bS^\top \nabla f(\x_t)}^2 
+
\frac{\alpha^2}{8} \norm{ \bS^\top \bH \bS }_2 \norm{\bv_t}^2.
\end{align*}

Combining above results, we can obtain that
\begin{align*}
&f(\x_{t+1})
\leq 
f(\x_t) - \frac{7\eta}{8\sigma_t} \norm{ \bS^\top \nabla f(\x_t) }^2 + \frac{\eta\alpha^2\norm{\bv_t}^2}{8\sigma_t} \\
&+ \frac{\eta^2}{\sigma_t^2} \cdot \frac{L_0 + L_1 \norm{\nabla f(\x_t)}}{2} \cdot \left( 2\norm{ \bS^\top \bH \bS}_2 \cdot \norm{\bS^\top \nabla f(\x_t)}^2 +  \frac{\alpha^2}{8} \norm{ \bS^\top \bH \bS }_2 \norm{\bv_t}^2 \right).
\end{align*}
Combining with event $\cE_{t,3}$, we can obtain the final result.
\end{proof}

\begin{lemma}
Let the function $f:\BR^d\to\BR$ satisfies Assumption~\ref{ass:LL} and events $\cE_{t,1}$, $\cE_{t,2}$, $\cE_{t,3}$, and $\cE_{t,4}$ hold. 
We set step size $\eta \leq {3\alpha}/{(64 L_1 \zeta_2 \sqrt{\ell - 1})}$ and the smooth parameter $\alpha^2 \leq \min\{ \frac{1}{2L_1^2\zeta_1},\; \frac{3\norm{\nabla f(\x_t)}^2}{8L_0^2\zeta_1}\}$ for Algorithm~\ref{alg:SA}, 
it holds that
\begin{equation}\label{eq:ff}
\begin{aligned}
	f(\x_{t+1})
	\leq&
	f(\x_t) 
	- \frac{\eta}{2\sigma_t} \norm{\nabla f(\x_t)}^2
	+ \frac{3\eta^2 L_0 \zeta_2\norm{\nabla f(\x_t) }^2}{2\sigma_t^2}
	+ \frac{\eta\alpha^2\norm{\bv_t}^2}{8\sigma_t}\\
	&
	+ \frac{\eta^2 \alpha^2\zeta_2 (L_0 + L_1 \norm{\nabla f(\x_t)}) \norm{\bv_t}^2}{16\sigma_t^2}.
\end{aligned}
\end{equation}
\end{lemma}
\begin{proof}
First, we have
\begin{align*}
	&f(\x_{t+1})
	\stackrel{\eqref{eq:ff_dec}}{\leq} 
	f(\x_t) - \frac{7\eta}{8\sigma_t} \norm{ \bS^\top \nabla f(\x_t) }^2 + \frac{\eta\alpha^2\norm{\bv_t}^2}{8\sigma_t} \\
	&+ \frac{\eta^2}{\sigma_t^2} \cdot \frac{L_0 + L_1 \norm{\nabla f(\x_t)}}{2} \cdot \left( 2\zeta_2 \cdot \norm{\bS^\top \nabla f(\x_t)}^2 +  \frac{\alpha^2\zeta_2}{8}  \norm{\bv_t}^2 \right)\\
	\stackrel{\eqref{eq:sig_low}}{\leq}&
	f(\x_t) 
	- \frac{7\eta}{8\sigma_t} \norm{ \bS^\top \nabla f(\x_t) }^2  
	+ \frac{\eta^2 \zeta_2 L_1 \norm{\nabla f(\x_t)} \norm{\bS^\top \nabla f(\x_t)}^2 }{\sigma_t  \frac{\alpha \norm{\nabla f(\x_t)}}{8\sqrt{\ell - 1}} } \\
	&
	+ \frac{\eta^2 L_0 \zeta_2\norm{ \bS^\top \nabla f(\x_t) }^2}{\sigma_t^2}
	+ \frac{\eta\alpha^2\norm{\bv_t}^2}{8\sigma_t}
	+ \frac{\eta^2 \alpha^2\zeta_2 (L_0 + L_1 \norm{\nabla f(\x_t)}) \norm{\bv_t}^2}{16\sigma_t^2}\\
	\leq&
	f(\x_t) 
	- \frac{7\eta}{8\sigma_t} \norm{ \bS^\top \nabla f(\x_t) }^2  
	+ \frac{\eta^2 \zeta_2 L_1 \norm{\nabla f(\x_t)} \norm{\bS^\top \nabla f(\x_t)}^2 }{\sigma_t  \frac{\alpha \norm{\nabla f(\x_t)}}{8\sqrt{\ell - 1}}  } \\
	&
	+ \frac{\eta^2 L_0 \zeta_2\norm{ \bS^\top \nabla f(\x_t) }^2}{\sigma_t^2}
	+ \frac{\eta\alpha^2\norm{\bv_t}^2}{8\sigma_t}
	+ \frac{\eta^2 \alpha^2\zeta_2 (L_0 + L_1 \norm{\nabla f(\x_t)}) \norm{\bv_t}^2}{16\sigma_t^2}\\
	=&
	f(\x_t) 
	- \frac{\eta}{\sigma_t} \left( \frac{7}{8} -  \frac{8\eta L_1 \zeta_2\sqrt{\ell - 1} }{\alpha}\right)\norm{\bS^\top \nabla f(\x_t)}^2  \\
	&
	+ \frac{\eta^2 L_0 \zeta_2\norm{ \bS^\top \nabla f(\x_t) }^2}{\sigma_t^2}
	+ \frac{\eta\alpha^2\norm{\bv_t}^2}{8\sigma_t}
	+ \frac{\eta^2 \alpha^2\zeta_2 (L_0 + L_1 \norm{\nabla f(\x_t)}) \norm{\bv_t}^2}{16\sigma_t^2}\\
	\leq&
	f(\x_t) 
	- \frac{\eta}{\sigma_t} \norm{\bS^\top \nabla f(\x_t)}^2
	+ \frac{\eta^2 L_0 \zeta_2\norm{ \bS^\top \nabla f(\x_t) }^2}{\sigma_t^2}
	+ \frac{\eta\alpha^2\norm{\bv_t}^2}{8\sigma_t}\\
	&
	+ \frac{\eta^2 \alpha^2\zeta_2 (L_0 + L_1 \norm{\nabla f(\x_t)}) \norm{\bv_t}^2}{16\sigma_t^2}, 
\end{align*}
where the last inequality is because of $\eta \leq \frac{3\alpha}{64 L_1 \zeta_2 \sqrt{\ell - 1}}$.

Combining with Eq.~\eqref{eq:E2}, we can obtain that
\begin{align*}
f(\x_{t+1})
\leq&
f(\x_t) 
- \frac{\eta}{2\sigma_t} \norm{\nabla f(\x_t)}^2
+ \frac{3\eta^2 L_0 \zeta_2\norm{\nabla f(\x_t) }^2}{2\sigma_t^2}
+ \frac{\eta\alpha^2\norm{\bv_t}^2}{8\sigma_t}\\
&
+ \frac{\eta^2 \alpha^2\zeta_2 (L_0 + L_1 \norm{\nabla f(\x_t)}) \norm{\bv_t}^2}{16\sigma_t^2}.
\end{align*}

\end{proof}

\begin{theorem}\label{thm:main}
	Suppose that the function $f:\BR^d\to\BR$ satisfies Assumptions~\ref{ass:LL} and \ref{ass:lower-bounded}
	and $\bS_t \in\RR^{d \times \ell}$ is a $({1}/{4}, k, \delta)$-oblivious sketching matrix. 
	Let $\cE_t = \cap_{i=1}^4 \;\cE_{t,i} $ be the event that all events $\cE_{t,1}$, $\cE_{t,2}$, $\cE_{t,3}$, and $\cE_{t,4}$ (defined in Eq.~\eqref{eq:E1}-Eq.~\eqref{eq:E4})  hold.
	Set the smooth parameter $\alpha$ and step size $\eta$ satisfy
	\begin{align*}
		\alpha^2 \leq \min\left\{ \frac{1}{2L_1^2\zeta_1 T},\; \frac{3\norm{\nabla f(\x_t)}^2}{8L_0^2\zeta_1}\right\} 
		\qquad
		\mbox{and}
		\qquad
		\eta = \frac{\alpha}{\sqrt{(\ell - 1)\zeta_2 T} },
	\end{align*}
with $T>0$ being large enough that $\eta \leq {3\alpha}/{(64 L_1 \zeta_2 \sqrt{\ell - 1})}$ for Algorithm~\ref{alg:SA}, then the sequence $\{\x_t\}$ generated satisfies 
\begin{equation}\label{eq:main}
	\begin{aligned}
		\frac{1}{T} \sum_{t=0}^{T-1}	 \norm{\nabla f(\x_t)}
		\leq&
		\frac{26\sqrt{\zeta_2}\Big( f(\x_0) - f(\x_T) \Big)}{\sqrt{T}} 
		+ \frac{96\cdot 26 L_0 \sqrt{\zeta_2}}{\sqrt{T}}
		+ \frac{104 + 52\sqrt{3}}{L_1 T} \\
		&
		+\frac{26\cdot 32\cdot  \sqrt{\zeta_2}  }{T^{3/2}}  
		+\frac{24\cdot 26\cdot  \zeta_2  }{T^{3/2}}.
	\end{aligned}
\end{equation}
Furthermore, it holds that
\begin{equation}\label{eq:P1}
\PP\left(\cap_{t=0}^{T-1} \;\cE_t\right) 
\geq 
1 - 4 T\delta.
\end{equation}
\end{theorem}
\begin{proof}
	First, we have
\begin{align*}
f(\x_{t+1}) 
\stackrel{\eqref{eq:ff}\eqref{eq:v_u}}{\leq}&
f(\x_t)
- \frac{\eta}{2\sigma_t} \norm{\nabla f(\x_t)}^2
+ \frac{3\eta^2 L_0 \zeta_2\norm{\nabla f(\x_t) }^2}{2\sigma_t^2}
\\
&
+ \frac{\eta\alpha^2 (L_0 + L_1 \norm{\nabla f(\x_t)}) \zeta_1}{\sigma_t}
+ \frac{\eta^2 \alpha^2\zeta_2 (L_0 + L_1 \norm{\nabla f(\x_t)})^2 \zeta_1}{2\sigma_t^2}\\
\leq&
f(\x_t) 
- \frac{\eta}{2\sigma_t} \norm{\nabla f(\x_t)}^2
+ \frac{3\eta^2 L_0 \zeta_2\norm{ \nabla f(\x_t) }^2}{2\sigma_t^2}
+ \frac{\eta\alpha^2  L_1 \norm{\nabla f(\x_t)} \zeta_1}{\sigma_t}\\
&
+ \frac{\eta^2 \alpha^2 \zeta_1 \zeta_2  L_1^2 \norm{\nabla f(\x_t)}^2 }{\sigma_t^2}
+ \frac{\eta L_0 \zeta_1 \alpha^2   }{\sigma_t}
+ \frac{\eta^2 \alpha^2 \zeta_1 \zeta_2 L_0^2  }{\sigma_t^2}\\
\stackrel{\eqref{eq:sig_low}}{\leq}&
f(\x_t) 
- \frac{\eta}{2\sigma_t} \norm{\nabla f(\x_t)}^2
+ \frac{96(\ell - 1)\eta^2 L_0 \zeta_2}{\alpha^2}
+ 8\eta \sqrt{\ell - 1} L_1 \zeta_1 \alpha \\
&
+ 64\eta^2 (\ell - 1) \zeta_1\zeta_2 L_1^2 
+ \frac{\eta L_0 \zeta_1 \alpha^2   }{\sigma_t}
+ \frac{\eta^2 \alpha^2 \zeta_1 \zeta_2 L_0^2  }{\sigma_t^2}\\
\stackrel{\eqref{eq:sig_up}}{\leq}&
f(\x_t) 
- \frac{\eta \sqrt{\ell - 1}}{26\alpha }  \norm{\nabla f(\x_t)}
+ \frac{96(\ell - 1)\eta^2 L_0 \zeta_2}{\alpha^2}
+ 8\eta \sqrt{\ell - 1} L_1 \zeta_1 \alpha \\
&
+ 64\eta^2 (\ell - 1) \zeta_1\zeta_2 L_1^2 
+ \frac{\eta L_0 \zeta_1 \alpha^2   }{\sigma_t}
+ \frac{\eta^2 \alpha^2 \zeta_1 \zeta_2 L_0^2  }{\sigma_t^2}.
\end{align*}

Reformulate above equation as
\begin{align*}
\frac{\eta \sqrt{\ell - 1}}{26\alpha }  \norm{\nabla f(\x_t)}
\leq& 
f(\x_t) - f(\x_{t+1})
+ \frac{96(\ell - 1)\eta^2 L_0 \zeta_2}{\alpha^2}
+ 8\eta \sqrt{\ell - 1} L_1 \zeta_1 \alpha \\
&
+ 64\eta^2 (\ell - 1) \zeta_1\zeta_2 L_1^2 
+ \frac{\eta L_0 \zeta_1 \alpha^2   }{\sigma_t}
+ \frac{\eta^2 \alpha^2 \zeta_1 \zeta_2 L_0^2  }{\sigma_t^2}.
\end{align*}
By telescoping above equation, we can obtain 
\begin{align*}
\frac{\eta \sqrt{\ell - 1}}{26\alpha }  \sum_{t=0}^{T-1}	 \norm{\nabla f(\x_t)}
\leq& 
f(\x_0) - f(\x_T) + \frac{96(\ell - 1)\eta^2 L_0 \zeta_2}{\alpha^2} \cdot T 
+ 8\eta \sqrt{\ell - 1} L_1 \zeta_1 \alpha \cdot T \\
&
+ 64\eta^2 (\ell - 1) \zeta_1\zeta_2 L_1^2 \cdot T
+ \sum_{t=0}^{T-1} \left( \frac{\eta L_0 \zeta_1 \alpha^2   }{\sigma_t}
+ \frac{\eta^2 \alpha^2 \zeta_1 \zeta_2 L_0^2  }{\sigma_t^2}\right).
\end{align*}

Therefore,
\begin{align*}
	\frac{1}{T} \sum_{t=0}^{T-1}	 \norm{\nabla f(\x_t)}
	\leq& 
	\frac{26\alpha\Big( f(\x_0) - f(\x_T) \Big)}{\eta \sqrt{\ell - 1} } \cdot \frac{1}{T} 
	+ \frac{96\cdot 26 \sqrt{\ell - 1} \eta L_0 \zeta_2}{\alpha}
	+ 8\cdot 26 \cdot L_1 \zeta_1 \alpha^2\\
	&
	+26\cdot 64\cdot \eta\sqrt{\ell-1} \zeta_1\zeta_2 L_1^2 \alpha 
	+ \frac{26\alpha}{\sqrt{\ell - 1} T}\sum_{t=0}^{T-1} \left( \frac{ L_0 \zeta_1 \alpha^2   }{\sigma_t}
	+ \frac{\eta \alpha^2 \zeta_1 \zeta_2 L_0^2  }{\sigma_t^2}\right)\\
	=&
	\frac{26\Big( f(\x_0) - f(\x_T) \Big)\sqrt{\zeta_2}}{\sqrt{T}} 
	+ \frac{96\cdot 26 L_0 \sqrt{\zeta_2}}{\sqrt{T}}
	+ 8\cdot 26 \cdot L_1 \zeta_1 \alpha^2\\
	&
	+\frac{26\cdot 64\cdot  \zeta_1\sqrt{\zeta_2} L_1^2 \alpha^2 }{\sqrt{T}} 
	+ \frac{26\alpha}{\sqrt{\ell - 1} T}\sum_{t=0}^{T-1} \left( \frac{ L_0 \zeta_1 \alpha^2   }{\sigma_t}
	+ \frac{\eta \alpha^2 \zeta_1 \zeta_2 L_0^2  }{\sigma_t^2}\right),
\end{align*}
where the last equality is because of the setting that $\eta = \frac{\alpha}{\sqrt{\ell - 1} \sqrt{T}}$.

Furthermore,
\begin{align*}
	&\frac{\alpha}{\sqrt{\ell - 1}} \cdot \left(\frac{ L_0 \zeta_1 \alpha^2   }{\sigma_t} 
	+ \frac{\eta \alpha^2 \zeta_1 \zeta_2 L_0^2  }{\sigma_t^2}\right)\\
	\stackrel{\eqref{eq:sig_low}}{\leq}&
	\frac{\alpha}{\sqrt{\ell - 1}} 
	\cdot  \left(\frac{8\sqrt{\ell - 1}L_0 \zeta_1 \alpha}{\norm{\nabla f(\x_t)}}
	+ \frac{8^2 (\ell-1) \eta \zeta_1\zeta_2 L_0^2}{\norm{\nabla f(\x_t)}^2}\right)\\
	\leq&
	\frac{8L_0 \zeta_1 \alpha^2}{\norm{\nabla f(\x_t)}}
	+ \frac{8^2 \alpha^2 \zeta_1\zeta_2 L_0^2}{\norm{\nabla f(\x_t)}^2 \sqrt{T}}\\
	\leq&
	\frac{8L_0 \zeta_1 \sqrt{ \frac{1}{2L_1^2 \zeta_1}} \sqrt{ \frac{3\norm{\nabla f(\x_t)}^2}{8L_0^2\zeta_1} }  }{\norm{\nabla f(\x_t)}}
	+ \frac{8^2 \frac{3\norm{\nabla f(\x_t)}^2}{8L_0^2\zeta_1} \zeta_1\zeta_2 L_0^2}{\norm{\nabla f(\x_t)}^2 \sqrt{T}}\\
	=& 
	\frac{2\sqrt{3}}{L_1} + \frac{24\zeta_2}{\sqrt{T}},
\end{align*}
where the last inequality is because of the assumption that $\alpha^2 \leq \min\{ \frac{1}{2L_1^2\zeta_1},\; \frac{3\norm{\nabla f(\x_t)}^2}{8L_0^2\zeta_1}\}$.

Therefore,
\begin{align*}
\frac{1}{T} \sum_{t=0}^{T-1}	 \norm{\nabla f(\x_t)}
\leq& 
\frac{26\sqrt{\zeta_2}v\Big( f(\x_0) - f(\x_T) \Big)}{\sqrt{T}} 
+ \frac{96\cdot 26 L_0 \sqrt{\zeta_2}}{\sqrt{T}}
+ 8\cdot 26 \cdot L_1 \zeta_1 \alpha^2\\
&
+\frac{26\cdot 64\cdot  \zeta_1\sqrt{\zeta_2} L_1^2 \alpha^2 }{\sqrt{T}} 
+ \frac{52\sqrt{3}}{L_1 T} 
+ \frac{24 \cdot 26 \cdot\zeta_2}{T^{3/2}}\\
\leq&
\frac{26\sqrt{\zeta_2}\Big( f(\x_0) - f(\x_T) \Big)}{\sqrt{T}} 
+ \frac{96\cdot 26 L_0 \sqrt{\zeta_2}}{\sqrt{T}}
+ \frac{104}{L_1 T}  \\
&
+\frac{26\cdot 32\cdot  \sqrt{\zeta_2}  }{T^{3/2}} 
+ \frac{52\sqrt{3}}{L_1 T} 
+ \frac{24 \cdot 26 \cdot\zeta_2}{T^{3/2}}\\
=&
\frac{26\sqrt{\zeta_2}\Big( f(\x_0) - f(\x_T) \Big)}{\sqrt{T}} 
+ \frac{96\cdot 26 L_0 \sqrt{\zeta_2}}{\sqrt{T}}
+ \frac{104 + 52\sqrt{3}}{L_1 T} \\
&
+\frac{26\cdot 32\cdot  \sqrt{\zeta_2}  }{T^{3/2}}  
+\frac{24\cdot 26\cdot  \zeta_2  }{T^{3/2}}  
,
\end{align*}
where the second inequality is because of $\alpha^2 \leq \frac{1}{2L_1^2\zeta_1 T}$.

Next, we will prove the result in Eq.~\eqref{eq:P1}.
By a union bound, we have
\begin{align*}
	\PP\left(\cE_t \right) 
	\geq 
	1 -  \sum_{i=1}^4 \PP\left(\overline{\cE_{t,i}}\right)
	\geq 
	1 - \left( \delta + \delta + \delta + \exp(-\cO\left(\ell\right)) \right)
	\geq
	1 - 4\delta,
\end{align*}
where $\overline{\cE_{t,i}}$ denotes the complement of $\cE_{t,i}$, the second inequality is because of Lemma~\ref{lem:E1}-Lemma~\ref{lem:E4}, and the last inequality is because of $\ell = \cO\left( k\log\frac{1}{\delta} \right)$.

Using the union bound again, we can obtain that
\begin{align*}
\PP\left(\cap_{t=0}^{T-1} \;\cE_t\right) 
\geq 
1 - \sum_{t=0}^{T-1} \PP\left(\overline{\cE_t}\right)
\geq
1 - 4 T\delta,
\end{align*}
which concludes the proof.
\end{proof}

\begin{corollary}\label{cor:main}
Suppose that the function $f:\BR^d\to\BR$ satisfies Assumptions~\ref{ass:LL} and \ref{ass:lower-bounded}.
Let $\bS_t\in\RR^{d\times \ell}$ be the Gaussian or Rademacher sketching matrices with $\ell = 16k\log({4T}/{\delta})$ and $\delta\in(0,1)$.
Set the smooth parameter $\alpha$ and step size $\eta$ satisfy
\begin{align*}
	\alpha^2 \leq \min\left\{ \frac{1}{2L_1^2\zeta_1 T},\; \frac{3\varepsilon^2}{8L_0^2\zeta_1}\right\} 
	\qquad
	\mbox{and}
	\qquad
	\eta = \frac{\alpha}{\sqrt{(\ell - 1)\zeta_2 T} },
\end{align*}
with $T>0$ being large enough that $\eta \leq {3\alpha}/{(64 L_1 \zeta_2 \sqrt{\ell - 1})}$ and with $\varepsilon>0$ being the target precision for Algorithm~\ref{alg:SA}, then if the total iteration number $T$ satisfies 
\begin{equation}\label{eq:T}
	T = \cO\left( \frac{\norm{\bH}}{\varepsilon^2} + \frac{\tr(\bH)}{k\cdot \varepsilon^2} \right)
\end{equation}
with a probability at least $1-\delta$, the sequence $\{\x_t\}$ with $t = 0,\dots, T-1$ generated by Algorithm~\ref{alg:SA} satisfies that
\begin{align*}
	\min_{t=0,\dots, T-1} \norm{\nabla f(\x_t)} \leq \varepsilon.
\end{align*}
Furthermore, the total query complexity is 
\begin{equation}\label{eq:Q}
	Q = \cO\left( \left(\frac{k \norm{\bH}}{\varepsilon^2} +  \frac{\tr(\bH)}{\varepsilon^2}\right)\log\frac{1}{\varepsilon \delta}  \right).
\end{equation}
\end{corollary}
\begin{proof}
By the properties of Gaussian and Rademacher sketching matrices in Sec.~\ref{subsec:Sketch}, we can obtain that $\bS_t\in\RR^{d\times \ell}$ with $\ell = 16k\log\frac{T}{\delta}$ will lead to Eq.~\eqref{eq:P1} hold with $1 - \delta$.

By setting the right hand side of  Eq.~\eqref{eq:main} to $\varepsilon$, we can obtain that 
\begin{align*}
	T 
	= 
	\cO\left( \frac{\zeta_2}{\varepsilon^2} \right) 
	\stackrel{\eqref{eq:zeta_2}}{=}
	\cO\left( \frac{\norm{\bH}}{\varepsilon^2} + \frac{\tr(\bH)}{k\cdot \varepsilon^2} \right).
\end{align*}

The total query complexity is 
\begin{align*}
	Q 
	= 
	T \cdot \ell 
	= 
	\cO\left( \frac{\zeta_2}{\varepsilon^2} \cdot \ell \right) 
	=
	\cO\left( \left(\frac{k \norm{\bH}}{\varepsilon^2} +  \frac{\tr(\bH)}{\varepsilon^2}\right)\log\frac{1}{\varepsilon \delta}  \right),
\end{align*}
where the last equality is because of $\ell = \cO\left(k\log\frac{T}{\delta}\right)$.
\end{proof}

\begin{remark}\label{rmk:main}
Theorem~\ref{thm:main} and Corollary~\ref{cor:main} show Algorithm~\ref{alg:SA} can achieve an $\cO\left( {\norm{\bH}}/{\varepsilon^2} + {\tr(\bH)}/(k\varepsilon^2)\right)$ iteration complexity. 
If we increase $k$ which means increasing the batch size $\ell$, Algorithm~\ref{alg:SA} can find a $\varepsilon$-stationary point with less steps.
However Eq.~\eqref{eq:Q} implies that increasing the batch size $\ell$ will not reduce the query complexity at all.
\end{remark}
\begin{remark}\label{rmk:Q}
The query complexity in Eq.~\eqref{eq:Q} achieves a weak dimension dependency. 
To make this clear, we can reformulate Eq.~\eqref{eq:Q} as 
\begin{align*}
	Q = \cO\left(\left(\frac{k \norm{\bH}}{\varepsilon^2} +  \frac{\tr(\bH)}{\norm{\bH}} \cdot \frac{\norm{\bH}}{\varepsilon^2}\right)\log\frac{1}{\varepsilon \delta}\right).
\end{align*} 
In many real applications, we can choose proper matrix $\bH$ such that ${\tr(\bH)}/{\norm{\bH}} \ll d$.
Thus, Algorithm~\ref{alg:SA} can inherently achieve weak dimension dependency even in our extended $(L_0, L_1)$-smooth condition in Assumption \ref{ass:LL}.
Our work extend the weak dimension dependency of zeroth-order algorithms with sketching matrices  which holds for classical smooth functions to the  extended $(L_0, L_1)$-smooth functions.
\end{remark}

{
\begin{remark}\label{remark:deter-M}
Recall that even by taking $\hg(\x_t)=\nabla f(\x_t)$ and Assumption \ref{ass:LL} with $\bH=\bI$, finding an $\varepsilon$-stationary point via the fixed stepsize update
$\x_{t+1}=\x_t - \eta\hg(\x_t)$
with $\eta_t>0$ 
requires the iteration numbers that depends on 
$M=\sup\{||\nabla f(\x)|| \mid \x \text{~such that~} f(\x)\leq f(\x_0)\}$ \citep[Theorem 4]{zhang2019gradient}.
In contrast, our Corollary \ref{cor:main} shows that
the iteration numbers required by Algorithm~\ref{alg:SA} only depends on $\norm{\bH}$ and $\tr(\bH)$,
which implies the adaptive zeroth-order method can converge arbitrary faster than algorithm with fixed stepsize for the poor initialization.
\end{remark}}

\section{Convergence Analysis for the Stochastic Problem}

This section present the convergence analysis of Algorithm \ref{alg:SA1} for the deterministic problem (\ref{eq:obj1}).

\begin{assumption}\label{ass:hL}
	We assume that $f(\x;\xi)$ is twice differentiable and $\norm{\nabla^2 f(\x;\xi)} \leq \hL$ if $\x\in\cK$ where the set $\cK$ is defined as
	\begin{equation}
		\cK:=\left\{ \x \mid \norm{\x - \x_0} \leq \frac{3\sqrt{ T}}{350 L_1 \sqrt{\zeta_2}}  \right\}.
	\end{equation}
	And vector $\x_0$ is initial point of Algorithm~\ref{alg:SA1}, $T$ is the total iteration number, and $\zeta_2$ is defined in Eq.~\eqref{eq:zeta_2}. 
\end{assumption}
Above assumption is only used to bound the value of  smooth parameter $\alpha$.

\subsection{Events}

Assume the objective function $f(\x)$ satisfies Assumption~\ref{ass:LL} and $\bS_t$ is an  oblivious $\left({1}/{4}, k, \delta\right)$-random sketching matrix just as defined in Definition~\ref{def:ske} with $k\geq 4 $ and $\delta\in(0,1)$. 
Then, we will define the following two events:
\begin{align}
	&\hE_{t,1} := \left\{ \sum_{i=1}^{\ell} \norm{\si}^4  \leq \hzt_1  \mid \si \mbox{ is the } i\mbox{-th column of } \bS_t \right\}, \label{eq:E11}\\
	&\hE_{t,2} := \left\{ \left(\frac{3}{4} - \frac{1}{k}\right) \norm{\nabla f(\x_t;\;\xi_t)}^2 
	\leq 	\norm{\bS_t^\top \nabla f(\x_t;\;\xi_t) }^2
	\leq \left(\frac{5}{4} + \frac{1}{k}\right)  \norm{\nabla f(\x_t;\;\xi_t)}^2 \right\}, \label{eq:E21}\\
	&\hE_{t,3} := \left\{\frac{1}{\ell}  \left( \sum_{i=1}^{\ell} \dotprod{\nabla f(\x_t;\;\xi_t), \si} \right)^2 \leq \frac{1}{4}\norm{\nabla f(\x_t;\;\xi_t)}^2\right\}, \label{eq:E31}\\
	&\hE_{t,4} := \left\{ \norm{\bS_t^\top \left(\nabla f(\x_t;\;\xi_t) - \nabla f(\x_t) \right)}^2
	\leq \left(\frac{5}{4} + \frac{1}{k}\right)  \norm{\nabla f(\x_t;\;\xi_t) - \nabla f(\x_t)}^2 \right\} \label{eq:E41},
\end{align} 
where $\hzt_1$ is defined in Eq.~\eqref{eq:hzt}.

Next, we will bound the probabilities of above events happen.
\begin{lemma}\label{lem:E11}
	Letting $\bS_t$ be an  oblivious $\left({1}/{4}, k, \delta\right)$-random sketching matrix just as defined in Definition~\ref{def:ske} with $k\geq 4 $ and $\delta \in(0,1)$, then event $\hE_{t,1}$ holds with a probability at least $1-\delta$ with $\hzt_1$ defined as	
	\begin{equation}\label{eq:hzt}
		\hzt_1 :=  \left(\frac{5\ell}{4} + \frac{d \ell}{4k}\right)^2.
	\end{equation}
\end{lemma}
\begin{proof}
	We have
	\begin{align*}
	    \sum_{i=1}^{\ell} \norm{\si}^4
		\leq 
		\norm{\bS_t}_F^4
		\stackrel{\eqref{eq:S_norm}}{\leq}
		\left(\frac{5\ell}{4} + \frac{d \ell}{4k}\right)^2,
	\end{align*}
	where the last inequality is because of Lemma~\ref{lem:bound_S_F_norm}.
\end{proof}

\begin{lemma}\label{lem:E21}
	Letting $\bS\in\RR^{\ell\times d}$ be an oblivious $\left({1}/{4}, k, \delta\right)$-random sketching matrix just as defined in Definition~\ref{def:ske} with $k\geq 4 $ and $\delta \in(0,1)$, then event $\hE_{t,2}$ holds with  a probability at least $1-\delta$.
\end{lemma}
\begin{proof}
	The proof is almost the same as that of Lemma~\ref{lem:E2} but replacing $\nabla f(\x)$ with $\nabla f(\x;\;\xi)$.
\end{proof}

\begin{lemma}\label{lem:E31}
	If $\bS\in\RR^{d\times \ell}$ is random sketching matrix which is Gaussian, Rademacher, SRHT, or Sparse Embedding (refer to Sec.~\ref{subsec:Sketch}), then event $\hE_{t,3}$ defined in Eq.~\eqref{eq:E31} holds  with  a probability at least $1 - \delta'$ with $\delta' = \exp\left(-\cO(\ell)\right)$.
\end{lemma}
\begin{proof}
	Refer to Lemma~\ref{lem:ss} with $\nabla f(\x)$ replaced by $\nabla f(\x;\;\xi)$.
\end{proof}

\begin{lemma}\label{lem:E41}
	Letting $\bS\in\RR^{\ell\times d}$ is an oblivious $\left({1}/{4}, k, \delta\right)$-random sketching matrix just as defined in Definition~\ref{def:ske} with $k\geq 4 $ and $\delta \in(0,1)$, then event $\hE_{t,4}$ holds with  a probability at least $1-\delta$.
\end{lemma}
\begin{proof}
	The proof is almost the same as that of Lemma~\ref{lem:E2} but replacing $\nabla f(\x)$ with $\nabla f(\x;\xi) - \nabla f(\x)$.
\end{proof}

\subsection{Properties of Estimated Gradient and Standard Deviation}

\begin{lemma}\label{lem:hg_prop}
	Given  $\bS\in\RR^{d \times \ell}$ being a pre-defined matrix, let each function $f(\x;\xi)$ satisfy Assumption~\ref{ass:hL} with all $\x$ and $\x + \alpha \si$ belonging to $\cK$. 
	Conditioned on event $\hE_{t,1}$ defined in Eq.~\eqref{eq:E11},  the approximate stochastic gradient $\hg(\x)$ defined in Eq.~\eqref{eq:g1} satisfies that
	\begin{equation}\label{eq:hg_prop}
		\hg(\x) = \bS\bS^\top \nabla f(\x;\xi) + \frac{\alpha}{4}\bS\hv \mbox{ and } \norm{\hv}^2 \le  \hL^2 \hzt_1, 
	\end{equation} 
	where $\hv$ is an $\ell$-dimension vector with $\hat{v}^{(i)} = \si^\top \nabla^2 f(\x_{\si};\xi) \si$ being its $i$-th entry, $\si$ is the $i$-th column of $\bS$, $\x_{\si}$ is a point lying in the line between $\x$ and $\x + \si$,  
	  and $\hzt_1$ is defined in Eq.~\eqref{eq:hzt}.
\end{lemma}
\begin{proof}
First, by the Taylor's expansion, we can obtain that
	\begin{align*}
		f(\x+\alpha \si;\;\xi) = f(\x ;\;\xi) + \alpha \dotprod{\nabla f(\x;\;\xi), \si} + \frac{\alpha^2}{2} \si^\top \nabla^2 f(\x_{\si};\;\xi) \si.
	\end{align*}
	Thus, we have
	\begin{align*}
		&\sum_{i=1}^{\ell}\frac{f(\x+\alpha \si;\;\xi) - f(\x;\;\xi)}{\alpha}\si \\
		=& 
		\sum_{i=1}^{\ell} \frac{\alpha \dotprod{\nabla f(\x;\;\xi), \si} + \frac{\alpha^2}{2} \si^\top \nabla^2 f(\x_{\si};\;\xi) \si }{\alpha} \si\\
		=&
		\sum_{i=1}^{\ell} \dotprod{\nabla f(\x;\;\xi), \si} \si 
		+ \frac{\alpha }{2} \sum_{i=1}^{\ell} \left( \si^\top \nabla^2 f(\x_{\si};\;\xi) \si  \right) \si \\
		=& \bS\bS^\top \nabla f(\x;\;\xi) + \frac{\alpha}{4} \bS \hv
	\end{align*}

	Furthermore, by the definition of $\hv$, we can obtain that
	\begin{align*}
		|\hv^{(i)}| = |\si^\top \nabla^2 f(\x_{\si};\;\xi) \si | \leq \hL \norm{\si}^2,
	\end{align*}
	where the last inequality is because of Assumption~\ref{ass:hL}.
	
	Therefore,
	\begin{align*}
		\norm{\hv}^2 = \sum_{i=1}^{\ell} |\hv^{(i)}|^2 \leq \hL^2 \sum_{i=1}^{\ell} \norm{\si}^4 \leq \hL^2 \hzt_1,
	\end{align*}
	where the last inequality is because of the condition that event $\hE_{t,1}$ holds. 
\end{proof}

\begin{lemma}\label{lem:sig1}
Letting $\sigma_t$ be defined in Eq.~\eqref{eq:sig1}, then it holds that
	\begin{align*}
		\sigma_t^2 
		\leq 
		\frac{3\alpha^2}{2(\ell-1)} \sum_{i=1}^{\ell} \dotprod{ \nabla f(\x_t;\;\xi_t), \bm{s}_i }^2 
		+ \frac{3}{\ell - 1} 	\sum_{i=1}^{\ell}  \hgam(\x_t, \alpha \bm{s}_j)^2,
	\end{align*}
	and
	\begin{align*}
		\sigma_t^2 
		\geq 
		\frac{\alpha^2}{2(\ell-1)} \sum_{i=1}^{\ell} \left(\dotprod{ \nabla f(\x_t;\;\xi_t), \bm{s}_i }^2  - \frac{1}{\ell^2}  \left( \sum_{i=1}^{\ell} \dotprod{\nabla f(\x_t;\;\xi_t), \si} \right)^2 \right) - \frac{1}{\ell - 1} 	\sum_{i=1}^{\ell}  \hgam(\x_t, \alpha \bm{s}_j)^2,
	\end{align*}
	where we define 
	\begin{equation}\label{eq:gamma1}
		\hgam(\x, \alpha \si) = f(\x + \alpha \si;\xi) - f(\x;\xi) - \alpha \dotprod{\nabla f(\x;\xi), \si}.
	\end{equation} 
\end{lemma}
\begin{proof}
The proof is the same as that of Lemma~\ref{lem:sig}.
\end{proof}

\begin{lemma}\label{lem:gamma1}
Let $\hgam(\x, \alpha \si)$ be defined in Eq.~\eqref{eq:gamma1}. Conditioned on event $\hE_{t,1}$ and Assumption~\ref{ass:hL} with all $\x$ and $\x + \alpha \si$ belonging to $\cK$, then	it holds that
	\begin{align*}
		\sum_{i=1}^{\ell} \hgam(\x, \alpha \si)^2 
		\leq 
		\frac{\alpha^4 \hL^2 \hzt_1 }{4},
	\end{align*}
	where $\hzt_1$ is defined in Eq.~\eqref{eq:hzt}.
\end{lemma}
\begin{proof}
	By the definition of $\hat{\gamma}(\cdot,\cdot)$ in Eq.~\eqref{eq:gamma1}, we have
	\begin{align*}
		\hgam(\x, \alpha \si) = f(\x + \alpha \si; \xi) - f(\x;\xi) - \alpha \dotprod{\nabla f(\x;\xi), \si} = \frac{\alpha^2}{2} \si^\top \nabla^2 f(\x_{\si}';\xi) \si
	\end{align*}
	where $\xi_{\si}'$ is a point lying the line between $\x$ and $\x + \alpha \si$.
	By Assumption~\ref{ass:hL}, we can obtain that
	\begin{align*}
		|\si^\top \nabla^2 f(\x_{\si}') \si| \leq \hL \norm{\si}^2.
	\end{align*}
	Therefore, we can obtain that
	\begin{align*}
		\sum_{i=1}^{\ell} \hgam(\x, \alpha \si)^2 
	\leq 
	\frac{\alpha^4 \hL^2}{4} \sum_{i=1}^{\ell} \norm{\si}^4 = \frac{\alpha^4 \hL^2 \hzt_1 }{4}. 
	\end{align*}
\end{proof}

\begin{lemma}\label{lem:hsig}
	Assuming that $\beta = {\alpha \nu}/{(4\sqrt{\ell - 1})}$ and  $\alpha^2 \le {\nu^2}/{(8\hL^2 \hzt_1)}$ with $\hzt_1$ defined in Eq.~\eqref{eq:hzt},  conditioned on event $\hE_{t,3}$ defined in Eq.~\eqref{eq:E31}, then it holds that
	\begin{align}
		\sigma_t + \beta \geq& \sqrt{\frac{\alpha^2}{8(\ell - 1)}  \cdot \norm{\nabla f(\x_t;\;\xi_t)}^2  + \frac{\beta^2}{2}},\label{eq:sig_low1}\\
		\sigma_t \leq& \frac{3\alpha}{2\sqrt{\ell-1}}\norm{\nabla f(\x_t;\;\xi_t)} + \sqrt{ \frac{3\alpha^4 \hL^2\hzt_1}{4(\ell - 1)} }. \label{eq:sig_up1}
	\end{align}
\end{lemma}
\begin{proof}
	Combining Lemma~\ref{lem:sig1} with  event $\hE_{t,3}$, we can obtain that 
	\begin{align*}
		\sigma_t^2 
		\geq
		\frac{\alpha^2}{2(\ell-1)} \cdot \frac{\norm{\nabla f(\x_t;\;\xi_t)}^2}{4} - \frac{1}{\ell - 1} 	\sum_{i=1}^{\ell}  \hgam(\x, \alpha \bm{s}_j)^2.
	\end{align*}
	
	By Lemma~\ref{lem:gamma1}, we can obtain that
	\begin{align*}
		\sigma_t^2 
		\geq&
		\frac{\alpha^2}{8(\ell - 1)}  \cdot \norm{\nabla f(\x_t;\;\xi_t)}^2 - \frac{\alpha^4 \hL^2  \hzt_1}{4(\ell-1)}.
	\end{align*}
	Furthermore,
	\begin{align*}
	(\sigma_t + \beta)^2 
	\geq& 
	\sigma_t^2 + \beta^2  
	\ge 
	\frac{\alpha^2}{8(\ell - 1)}  \cdot \norm{\nabla f(\x_t;\;\xi_t)}^2 - \frac{\alpha^4 \hL^2  \hzt_1}{4(\ell-1)} + \beta^2\\
	=&
	\frac{\alpha^2}{8(\ell - 1)}  \cdot \norm{\nabla f(\x_t;\;\xi_t)}^2 
	- \frac{\alpha^4 \hL^2  \hzt_1}{4(\ell-1)} 
	+ \frac{\alpha^2 \nu^2}{16(\ell-1)}\\
	\geq& 
	\frac{\alpha^2}{8(\ell - 1)}  \cdot \norm{\nabla f(\x_t;\;\xi_t)}^2  + \frac{\beta^2}{2}, 
	\end{align*}
	where the last inequality is because of $\alpha^2 \le \frac{\nu^2}{8\hL^2 \hzt_1}$
	
	Similarly, combining Lemma~\ref{lem:sig1} with Lemma~\ref{lem:gamma1},   we have
	\begin{align*}
		\sigma_t^2 
		\leq& 
		\frac{3\alpha^2}{2(\ell-1)} \sum_{i=1}^{\ell} \dotprod{ \nabla f(\x_t;\;\xi_t), \bm{s}_i }^2 
		+ \frac{3}{\ell - 1} 	\sum_{i=1}^{\ell}  \hgam(\x_t, \alpha \bm{s}_j)^2\\
		\stackrel{\eqref{eq:E21}}{\leq}&
		\frac{9\alpha^2}{4(\ell-1)}\norm{\nabla f(\x_t;\;\xi_t)}^2
		+ \frac{3}{\ell - 1} 	\sum_{i=1}^{\ell}  \hgam(\x_t, \alpha \bm{s}_j)^2\\
		\leq&
		\frac{9\alpha^2}{4(\ell-1)}\norm{\nabla f(\x_t;\;\xi_t)}^2
		+ \frac{3\alpha^4 \hL^2\hzt_1}{4(\ell - 1)},
	\end{align*}
	which implies that
	\begin{align*}
		\sigma_t
		\leq 
		\frac{3\alpha}{2\sqrt{\ell-1}}\norm{\nabla f(\x_t;\;\xi_t)} + \sqrt{ \frac{3\alpha^4 \hL^2\hzt_1}{4(\ell - 1)} }.
	\end{align*}
\end{proof}

\begin{lemma}
Let events $\hE_{t,1}$, $\hE_{t,2}$, and $\cE_{t,3}$ hold. 
Setting $\beta = \frac{\alpha \nu}{4\sqrt{\ell - 1}}$ and  $\alpha^2 \le \frac{\nu^2}{8\hL^2 \hzt_1}$, then it holds that
	\begin{equation}\label{eq:up}
		\frac{\norm{\hg(\x_t)}_{\bH}}{\sigma_t + \beta } 
		\leq 
		\frac{2\sqrt{6\zeta_2(\ell-1)}}{\alpha} + \frac{2 \hL \sqrt{(\ell-1)\zeta_2\hzt_1}}{\nu}.
	\end{equation}
\end{lemma}
\begin{proof}
First, we have
	\begin{align*}
		\frac{\norm{\hg(\x_t)}_{\bH}}{\sigma_t + \beta } 
		\stackrel{\eqref{eq:sig_low1}}{\leq}&
		\frac{\norm{\hg(\x_t)}_{\bH}}{\sqrt{\frac{\alpha^2}{8(\ell - 1)}  \cdot \norm{\nabla f(\x_t;\;\xi_t)}^2  + \frac{\beta^2}{2}}}
		\leq
		\frac{4\sqrt{\ell-1} \norm{\hg(\x_t)}_{\bH}}{\alpha \norm{\nabla f(\x_t;\;\xi_t)} + 4\sqrt{\ell - 1} \beta} \\
		\stackrel{\eqref{eq:hg_prop}}{\leq}&
		\frac{4\sqrt{\ell-1} \left(  \norm{\bS\bS^\top \nabla f(\x_t;\;\xi_t)}_{\bH} + \frac{\alpha}{2}\norm{\bS\hv}_{\bH}\right)}{\alpha \norm{\nabla f(\x_t;\;\xi_t)} + 4\sqrt{\ell - 1} \beta} \\
		\leq&
		\frac{\norm{\bS\bS^\top \nabla f(\x_t;\;\xi_t)}_{\bH}}{\frac{\alpha}{4\sqrt{(\ell-1)}} \norm{\nabla f(\x_t;\;\xi_t)}}
		+
		\frac{ \frac{\alpha}{2}\norm{\bS\hv}_{\bH}}{ \beta}.
	\end{align*}
	Furthermore,
	\begin{align*}
		\norm{\bS\bS^\top \nabla f(\x_t;\;\xi_t)}_{\bH}
		=& 
		\sqrt{ \nabla^\top f(\x_t;\;\xi_t) \bS\bS^\top \bH  \bS\bS^\top \nabla f(\x_t;\;\xi_t)}\\
		\stackrel{\eqref{eq:E3}}{\leq}&
		\sqrt{\zeta_2} \norm{ \bS^\top \nabla f(\x_t;\;\xi_t)} 
		\stackrel{\eqref{eq:E21}}{\leq}
		\sqrt{\frac{3\zeta_2}{2}} \norm{\nabla f(\x_t;\;\xi_t)}.
	\end{align*}
	Similarly, 
	\begin{align*}
		\norm{\bS\hv}_{\bH} 
		=
		\sqrt{ \hv^\top \bS^\top \bH \bS \hv }
		\leq
		\sqrt{\zeta_2} \norm{\hv}
		\stackrel{\eqref{eq:hg_prop}}{\leq}
		\hL \sqrt{\zeta_2\hzt_1}.
	\end{align*}
	
	Combining above results, we can obtain that
	\begin{align*}
		\frac{\norm{\hg(\x_t)}_{\bH}}{\sigma_t + \beta } 
		\le& 
		\frac{\sqrt{\frac{3\zeta_2}{2}} \norm{\nabla f(\x_t;\;\xi_t)}}{\frac{\alpha}{4\sqrt{(\ell-1)}} \norm{\nabla f(\x_t;\;\xi_t)}}
		+
		\frac{ \frac{\alpha}{2}\hL \sqrt{\zeta_2\hzt_1}}{ \beta}\\
		=&
		\frac{2\sqrt{6\zeta_2(\ell-1)}}{\alpha} + \frac{2 \hL \sqrt{(\ell-1)\zeta_2\hzt_1}}{\nu},
	\end{align*}
	where the last equality is because of $\beta = \frac{\alpha \nu}{4\sqrt{\ell - 1}}$.
\end{proof}

\subsection{Descent Analysis and Main Results}

\begin{lemma}
	Let functions $f(\x)$ and $f(\x;\;\xi)$ satisfy Assumption~\ref{ass:LL} and Assumption~\ref{ass:hL}. 
Matrix $\bS\in\RR^{d\times \ell}$ is an oblivious $(\frac{1}{4}, k, \delta)$-random sketching matrix. 
Assume that events $\hE_{t,1}$--$\hE_{t,3}$ and $\cE_{t,2}$--$\cE_{t,3}$ hold. 
Set $\beta = \frac{\alpha \nu}{4\sqrt{\ell - 1}}$,  $\alpha^2 \le \frac{\nu^2}{8\hL^2 \hzt_1}$, and the step size $\eta \leq \frac{\alpha }{ 700  \cdot \sqrt{\ell - 1} \zeta_2 L_1}$.
Then after one step update of Algorithm~\ref{alg:SA1}, it holds that
\begin{equation}\label{eq:dec}
\begin{aligned}
f(\x_{t+1})
\leq& 
f(\x_t)
- \frac{\eta\norm{\nabla f(\x_t)}^2}{ 4(\sigma_t + \beta) }
- \eta \cdot \frac{ \dotprod{\bS^\top \left(\nabla f(\x_t) - \nabla f(\x;\;\xi_t)\right), \bS^\top \nabla f(\x_t)} }{\sigma_t+\beta } \\
&
+
\frac{5\eta  \alpha^2 \hL^2 \hzt_1}{32\beta}
+ 
\frac{\eta^2 \zeta_2 (L_0 + L_1\norm{\nabla f(\x_t)})}{2}\cdot \left( \frac{48(\ell-1)}{\alpha^2} + \frac{\alpha^2 \hL^2 \hzt_1 }{8\beta^2}\right).
\end{aligned}
\end{equation}
\end{lemma}
\begin{proof}
First, by the update rule of Algorithm~\ref{alg:SA1},  we have
\begin{align*}
	\norm{\x_{t+1} - \x_t}_H 
	=& 
	\eta \frac{\norm{\hg(\x_t)}_{\bH}}{\sigma_t + \beta}
	\stackrel{\eqref{eq:up}}{\leq}
	\eta \left( \frac{2\sqrt{6\zeta_2(\ell-1)}}{\alpha} + \frac{2 \hL \sqrt{(\ell-1)\zeta_2\hzt_1}}{\nu} \right)\\
	=& 
	\frac{\sqrt{6}}{350L_1\sqrt{\zeta_2}}  + \frac{\alpha \hL \sqrt{\hzt_1}}{350L_1\sqrt{\zeta_2} \nu} 
	\leq \frac{\sqrt{6}}{350L_1\sqrt{\zeta_2}} + \frac{1}{350\sqrt{8}L_1\sqrt{\zeta_2}}
	\leq  \frac{1}{L_1 \sqrt{\zeta_2}} 
	\leq \frac{1}{L_1 \sqrt{\norm{\bH}}},
\end{align*}
where the second inequality is because the condition $\alpha^2 \le \frac{\nu^2}{8\hL^2 \hzt_1}$ and the last inequality is because of the definition of $\zeta_2$ in Eq.~\eqref{eq:zeta_2}.
Thus, the conditions in Assumption~\ref{ass:LL} hold.
Then, by the $(L_0,L_1)$-smoothness of $f(\x)$ and one step update of Algorithm~\ref{alg:SA1}, we have
\begin{align*}
	&f(\x_{t+1}) \\
	\stackrel{\eqref{eq:LL}}{\leq}&
	f(\x_t) + \dotprod{\nabla f(\x_t), \x_{t+1} - \x_t} + \frac{L_0 + L_1 \norm{\nabla f(\x_t)}}{2} ( \x_{t+1} - \x_t)^\top \bH  (\x_{t+1} - \x_t)\\
	=&
	f(\x_t) - \frac{\eta}{\sigma_t+\beta} \dotprod{\nabla f(\x_t), \hg(\x_t)} + \frac{\eta^2}{(\sigma_t+\beta)^2} \frac{L_0 + L_1 \norm{\nabla f(\x_t)}}{2} \hg(\x_t)^\top \bH \hg(\x_t)\\
	\stackrel{\eqref{eq:hg_prop}}{=}&
	f(\x_t) - \frac{\eta}{\sigma_t+\beta} \dotprod{\nabla f(\x_t), \bS\bS^\top \nabla f(\x_t;\;\xi_t) + \frac{\alpha}{4}\bS \hv_k} \\
	&+ \frac{\eta^2}{(\sigma_t+\beta)^2} \frac{L_0 + L_1 \norm{\nabla f(\x_t)}}{2} \left( \bS\bS^\top \nabla f(\x_t;\;\xi_t) + \frac{\alpha}{4}\bS \hv_k \right)^\top \bH \left( \bS\bS^\top \nabla f(\x_t;\;\xi_t) + \frac{\alpha}{4}\bS \hv_k \right).
\end{align*}

Furthermore,
\begin{align*}
&-  \dotprod{\nabla f(\x_t), \bS\bS^\top \nabla f(\x_t;\;\xi_t) + \frac{\alpha}{4}S \hv_k}\\
=&
-  \dotprod{\bS^\top \nabla f(\x_t),  \bS^\top \nabla f(\x_t)} 
- \dotprod{\bS^\top \nabla f(\x_t),\bS^\top  \left(\nabla f(\x_t;\;\xi_t)  - \nabla f(\x_t)\right) }
- \frac{\alpha}{4}\dotprod{ \bS^\top \nabla f(\x_t),\hv_k  } \\
\leq&
-\frac{\norm{\bS^\top \nabla f(\x_t)}^2}{2} - \dotprod{\bS^\top \nabla f(\x_t),\bS^\top  \left(\nabla f(\x_t;\;\xi_t)  - \nabla f(\x_t)\right) }
+ \frac{\alpha^2 \norm{\hv_k}^2}{8}.
\end{align*}

It also holds that
\begin{align*}
\left( \bS\bS^\top \nabla f(\x_t;\;\xi_t) + \frac{\alpha}{4}S \hv_k \right)^\top \bH \left( \bS\bS^\top \nabla f(\x_t;\;\xi_t) + \frac{\alpha}{4}S \hv_k \right)
\leq
\norm{\bS^\top \bH \bS}_2 \cdot \norm{ \bS^\top \nabla f(\x_t;\;\xi_t) + \frac{\alpha}{4} \hv_k }^2
\end{align*}

Therefore,
\begin{align*}
f(\x_{t+1})
\leq& 
f(\x_t)
- \frac{\eta\norm{\bS^\top\nabla f(\x_t)}^2}{ 2(\sigma_t+\beta) }
- \eta \cdot \frac{ \dotprod{\bS^\top \left(\nabla f(\x_t) - \nabla f(\x;\;\xi_t)\right), \bS^\top \nabla f(\x_t)} }{\sigma_t+\beta }\\
&
+ \frac{\eta\alpha^2}{8(\sigma_t+\beta)}  \norm{\hv_k}^2
+ \frac{\eta\alpha^2}{32(\sigma_t+\beta)}  \norm{\hv_k}^2\\
&
+ \frac{\eta^2}{(\sigma_t+\beta)^2} \frac{L_0 + L_1 \norm{\nabla f(\x_t)}}{2} \norm{\bS^\top \bH \bS}_2 \cdot \norm{ \bS^\top \nabla f(\x_t;\;\xi_t) + \frac{\alpha}{4} \hv_k }^2
\end{align*}

Combining Lemma~\ref{lem:hg_prop}, Lemma~\ref{lem:hsig} and event $\hE_{t,2}$ defined in Eq.~\eqref{eq:E21}, we can obtain that
\begin{align*}
&\frac{\left( \norm{ \bS^\top \nabla f(\x_t;\;\xi_t) } + \frac{\alpha}{4} \norm{\hv_k} \right) }{\sigma_t + \beta }
\stackrel{\eqref{eq:E21}\eqref{eq:sig_low1}\eqref{eq:hg_prop}}{\leq}
\frac{\sqrt{\frac{3}{2}} \norm{\nabla f(\x_t;\;\xi_t)} + \frac{\alpha \hL \sqrt{\hzt_1}}{4} }{\sqrt{\frac{\alpha^2}{8(\ell - 1)}  \cdot \norm{\nabla f(\x_t;\;\xi_t)}^2  + \frac{\beta^2}{2}}}\\
\leq&
\frac{\sqrt{\frac{3}{2}} \norm{\nabla f(\x_t;\;\xi_t)} }{\frac{\alpha \norm{\nabla f(\x_t;\;\xi_t)}}{4\sqrt{\ell-1}}}
+
\frac{\alpha \hL \sqrt{\hzt_1}}{4\beta}
= 
\frac{2\sqrt{6}\sqrt{\ell-1} }{\alpha} 
+ 
\frac{\alpha \hL \sqrt{\hzt_1}}{4\beta}.
\end{align*}

Therefore, we can obtain that
\begin{align*}
&\frac{\left( \bS\bS^\top \nabla f(\x_t;\;\xi_t) + \frac{\alpha}{4}S \hv_k \right)^\top H \left( \bS\bS^\top \nabla f(\x_t;\;\xi_t) + \frac{\alpha}{4}S \hv_k \right)
}{(\sigma_t+\beta)^2}\\
\leq& 
\zeta_2
\left( \frac{2\sqrt{6}\sqrt{\ell-1} }{\alpha} 
+ 
\frac{\alpha \hL \sqrt{\hzt_1}}{4\beta}.
 \right)^2
\leq 
\zeta_2\left( \frac{48(\ell-1)}{\alpha^2} + \frac{\alpha^2 \hL^2 \hzt_1 }{8\beta^2}\right),
\end{align*}
where the first inequality is also because of event $\cE_{t,3}$ defined in Eq.~\eqref{eq:E3}.

Furthermore,
\begin{align*}
\frac{\eta\alpha^2}{\sigma_t + \beta}  \norm{\hv_k}^2
\stackrel{\eqref{eq:hg_prop}}{\leq}
\frac{\eta\alpha^2 \hL^2 \hzt_1}{ \beta}.
\end{align*}

Therefore,
\begin{align*}
f(\x_{t+1})
\leq& 
f(\x_t)
- \frac{\eta\norm{\bS^\top\nabla f(\x_t)}^2}{ 2(\sigma_t + \beta) }
- \eta \cdot \frac{ \dotprod{\bS^\top \left(\nabla f(\x_t) - \nabla f(\x;\;\xi_t)\right), \bS^\top \nabla f(\x_t)} }{\sigma_t+\beta } \\
&
+
\frac{5\eta  \alpha^2 \hL^2 \hzt_1}{32\beta}
+ 
\frac{\eta^2 \zeta_2 (L_0 + L_1\norm{\nabla f(\x_t)})}{2}\cdot \left( \frac{48(\ell-1)}{\alpha^2} + \frac{\alpha^2 \hL^2 \hzt_1 }{8\beta^2}\right)\\
\leq&
f(\x_t)
- \frac{\eta\norm{\nabla f(\x_t)}^2}{ 4(\sigma_t + \beta) }
- \eta \cdot \frac{ \dotprod{\bS^\top \left(\nabla f(\x_t) - \nabla f(\x;\;\xi_t)\right), \bS^\top \nabla f(\x_t)} }{\sigma_t+\beta } \\
&
+
\frac{5\eta  \alpha^2 \hL^2 \hzt_1}{32\beta}
+ 
\frac{\eta^2 \zeta_2 (L_0 + L_1\norm{\nabla f(\x_t)})}{2}\cdot \left( \frac{48(\ell-1)}{\alpha^2} + \frac{\alpha^2 \hL^2 \hzt_1 }{8\beta^2}\right),
\end{align*}
where the last inequality is because of event $\cE_{t,2}$.
\end{proof}

\begin{lemma}
Let $\nabla f(\x;\;\xi)$ satisfy Assumption~\ref{ass:var}. 
Denote $\hE_t := \left(\cap_{i=1}^4 \hE_{t,i}\right) \cap \cE_{t,2} \cap \cE_{t,3}  $ which means events $\hE_{t,1}$--$\hE_{t,4}$ and $\cE_{t,2}$--$\cE_{t,3}$ hold. 
Assume $\hE_t$ hold for $t =0,\dots, T-1$. 
Setting  $\beta = \frac{\alpha \nu}{4\sqrt{\ell - 1}}$ and $\alpha^2 \le \frac{\nu^2}{8\hL^2 \hzt_1}$ with $\hzt_1$ defined in Eq.~\eqref{eq:hzt}, then with a probability at least $1-\delta$, it holds that  
	\begin{equation}\label{eq:avr}
		\frac{1}{T}\sum_{k=1}^{T}\frac{ \dotprod{\bS^\top \left(\nabla f(\x_t) - \nabla f(\x_t;\;\xi_t)\right), \bS^\top \nabla f(\x_t)} }{\sigma_t +\beta }
		\leq 
		\frac{1}{\sqrt{T}} \cdot \frac{9\sqrt{2} \sqrt{\ell - 1} \nu}{ \alpha} \sqrt{\log\frac{2}{\delta}}.
	\end{equation}
\end{lemma}
\begin{proof}
	First, by Lemma~\ref{lem:hsig}, we have
	\begin{equation}\label{eq:sig_beta}
		\begin{aligned}
			\frac{1}{\sigma_t+\beta }  
			\stackrel{\eqref{eq:sig_low1}}{\leq}& 	
			\frac{1}{ \sqrt{\frac{\alpha^2}{8(\ell - 1)}  \cdot \norm{\nabla f(\x_t;\;\xi_t)}^2  + \frac{\beta^2}{2}}}
			\leq
			\frac{\sqrt{2}}{\sqrt{\frac{\alpha^2}{8(\ell - 1)}  \cdot \norm{\nabla f(\x_t;\;\xi_t)}^2} + \sqrt{ \frac{\beta^2}{2} }}\\
			=&
			\frac{4\sqrt{\ell-1}}{\alpha \norm{\nabla f(\x_t;\;\xi_t)} + 4\sqrt{\ell - 1} \beta} 
			\leq 
			\frac{4\sqrt{\ell-1}}{\alpha    (\norm{\nabla f(\x_t)} - \nu + 4\sqrt{\ell - 1} \beta/\alpha) }
			=
			\frac{4\sqrt{\ell-1}}{\alpha    \norm{\nabla f(\x_t)} },
		\end{aligned}
	\end{equation}
	where the last equality is because of $\beta = \frac{\alpha \nu}{4\sqrt{\ell - 1}}$, the second inequality is because of the fact that $a^2 + b^2 \geq \frac{(a+b)^2}{2}$ for all $a, b\in\RR$, and the last inequality is because of Assumption~\ref{ass:var}.
	
	Then, we have
	\begin{align*}
		&\left|
		\frac{ \dotprod{\bS^\top \left(\nabla f(\x_t) - \nabla f(\x_t;\;\xi_t)\right), \bS^\top \nabla f(\x_t)} }{\sigma_t +\beta  }
		\right|\\
		\leq& 
		\frac{\norm{\bS^\top \left(\nabla f(\x_t) - \nabla f(\x_t;\;\xi_t)\right)  } \cdot \norm{ \bS^\top \nabla f(\x_t)} }{\sigma_t+\beta}\\
		\stackrel{\eqref{eq:E2}\eqref{eq:E41}}{\leq}&
		\frac{9\norm{\nabla f(\x_t) - \nabla f(\x_t;\;\xi_t)} \norm{\nabla f(\x_t)}}{4 (\sigma_t+\beta)}\\
		\stackrel{\eqref{eq:sig_beta}}{\leq}&
		\frac{9\sqrt{\ell - 1} \norm{\nabla f(\x_t) - \nabla f(\x_t;\;\xi_t)}}{\alpha }
		\leq 
		\frac{9\sqrt{\ell - 1} \nu}{\alpha},
	\end{align*}
	where the last inequality is because of Assumption~\ref{ass:var}. 
	
	By the Azuma-Hoeffding Inequality in Lemma~\ref{lem:concent} with $E_t = \frac{ \dotprod{\bS^\top \left(\nabla f(\x_t) - \nabla f(\x_t;\;\xi_t)\right), \bS^\top \nabla f(\x_t)} }{\sigma_t +\beta  }$, $B = \frac{9\sqrt{\ell - 1} \nu}{\alpha}$ and $\zeta =  \frac{1}{\sqrt{T}} \cdot \frac{9\sqrt{2} \sqrt{\ell - 1} \nu}{ \alpha} \sqrt{\log\frac{2}{\delta}}$, we can obtain that
	\begin{align*}
		\PP\left[\frac{1}{T}\sum_{t=0}^{T-1}\frac{ \dotprod{\bS^\top \left(\nabla f(\x_t) - \nabla f(\x;\;\xi_t)\right), \bS^\top \nabla f(\x_t)} }{\sigma_t +\beta } \geq \zeta\right]
		\leq
		\exp\left( - \frac{T \zeta^2}{2  \left(\frac{9\sqrt{\ell - 1} \nu}{\alpha} \right)^2} \right) = \delta,
	\end{align*}
	which concludes the proof.
\end{proof}

\begin{theorem}\label{thm:main1}
Let functions $f(\x)$ and $f(\x;\;\xi)$ satisfy Assumption~\ref{ass:LL}--\ref{ass:hL}. 
Matrix $\bS_t\in\RR^{d\times \ell}$ is an oblivious $({1}/{4}, k, \delta)$-random sketching matrix. 
Denote $\hE_t := \big(\cap_{i=1}^4 \hE_{t,i}\big) \cap \cE_{t,2} \cap \cE_{t,3} $ which means events $\hE_{t,1}$, $\hE_{t,1}$, $\hE_{t,4}$, and $\hE_{t,4}$ (defined in Eq.~\eqref{eq:E11}-Eq.~\eqref{eq:E41}) and $\cE_{t,2}$ and $\cE_{t,3}$ (defined in Eq.~\eqref{eq:E2}-Eq.~\eqref{eq:E3}) hold. 
Setting parameters of Algorithm~\ref{alg:SA1} as follow:
\begin{equation}
	\alpha = \min\left\{\frac{1}{T^{1/2}}, \sqrt{ \frac{\nu^2}{8\hL^2 \hzt_1} }\right\},\qquad \eta = \frac{\alpha }{ 700 \sqrt{T} \cdot \sqrt{\ell - 1} \zeta_2 L_1},\qquad\mbox{and}\qquad \beta = \frac{\alpha \nu}{4\sqrt{\ell - 1}},
\end{equation} 
then the sequence $\{\x_t\}$ with $t = 0,\dots, T-1$ generated by Algorithm~\ref{alg:SA1} satisfies the following property with a probability at least $1 - \delta$ and $0<\delta<1$,
\begin{equation}\label{eq:main1}
\begin{aligned}
\min_{t=0,\dots, T-1} &\norm{\nabla f(\x_t)} 
\leq \max\left\{\frac{1}{T^{1/4}}, \right.\\
&\left.\; 
\frac{28\cdot 700\cdot \nu \sqrt{\zeta_2} L_1 \cdot (f(\x_0) - f(\x^*))}{T^{1/4} }
+ \frac{252\sqrt{2}\nu^2}{ T^{1/4}}\sqrt{\log\frac{2}{\delta}} 
+ \frac{35   \hL^2\hzt_1 }{2 T^{1/4} }
+ \frac{  L_0 \nu \sqrt{\zeta_2}}{T^{1/4}L_1}
\right\}.
\end{aligned}
\end{equation}
Furthermore, it holds that
\begin{equation}\label{eq:P2}
\PP\left( \cap_{t=0}^{T-1}\; \hE_t \right) \geq 1 - 6T\delta.
\end{equation}
\end{theorem}
\begin{proof}

We first consider the case that  $\min\left\{\norm{\nabla f(\x_0)},\dots, \norm{\nabla f(\x_{T-1})}\right\} \leq \frac{1}{T^{1/4}}$. 
If this result holds, then we have proved the result in Eq.~\eqref{eq:main1}.
If not, we will prove the case that 
\begin{equation}\label{eq:upp}
	\min\left\{\norm{\nabla f(\x_0)},\dots, \norm{\nabla f(\x_{T-1})}\right\} \geq \frac{1}{T^{1/4}}.
\end{equation}	
For the notation convenience, we use $\phi$ to denote $1/T^{1/4}$.
Then, by Assumption~\ref{ass:var}, we can obtain that
\begin{align*}
	\sigma_t + \beta 
	\stackrel{\eqref{eq:sig_up1}}{\leq} 
	\sqrt{\frac{5}{2(\ell - 1)}} \alpha \norm{\nabla f(\x_t;\;\xi_t)} + \beta 
	\leq 
	\sqrt{\frac{5}{2(\ell - 1)}} \alpha \left(\norm{\nabla f(\x_t)} + \nu \right) + \beta,	
\end{align*}
where the last inequality is because of Assumption~\ref{ass:var}.

Thus, we can obtain
\begin{align*}
- \frac{\norm{\nabla f(\x_t)}^2}{\sigma_t + \beta}
\leq 
- \frac{\norm{\nabla f(\x_t)}^2}{\sqrt{\frac{5}{2(\ell - 1)}} \alpha \left(\norm{\nabla f(\x_t)} + \nu \right) + \beta}.
\end{align*}
Furthermore, 
\begin{align*}
	\frac{\norm{\nabla f(\x_t)}}{\sqrt{\frac{5}{2(\ell - 1)}}  \left(\norm{\nabla f(\x_t)} + \nu \right) + \beta/\alpha}  
	\stackrel{\eqref{eq:upp}}{\geq} 
	\frac{\phi}{\sqrt{\frac{5}{2(\ell - 1)}}  \left(\phi + \nu \right) + \beta/\alpha}
	\geq
	\frac{2 \sqrt{\ell - 1} \phi}{7 \nu},
\end{align*}
where the last inequality is because of the setting of $\beta$ such that
\begin{align*}
	\beta 
	= 
	\frac{\alpha \nu}{4\sqrt{\ell - 1}} 
	\leq 
	\alpha \cdot \left(\frac{7-\sqrt{10}}{2\sqrt{\ell - 1}} \nu - \frac{\sqrt{10}}{2\sqrt{\ell-1}}\phi\right)
	, \mbox{ with } 
	\phi  
	\leq
	\sqrt{\frac{5}{2}} - 1. 
\end{align*}

Furthermore, it holds that 
\begin{align*}
	\frac{48(\ell-1)}{\alpha^2} + \frac{\alpha^2 \hL^2 \hzt_1 }{8\beta^2} 
	= \frac{48(\ell-1)}{\alpha^2} + \frac{ 2(\ell - 1) \hL^2 \hzt_1 }{\nu^2}
	\leq
	\frac{50(\ell-1)}{\alpha^2},
\end{align*}
where the first equality is because of the value of $\beta$ and the last inequality is because of $\alpha^2 \leq \frac{\nu^2}{\hL^2 \hzt_1}$.

Combining with Eq.~\eqref{eq:dec} and the value of $\beta$, we can obtain that
\begin{align*}
	f(\x_{t+1}) 
	\leq& 
	f(\x_t) - \frac{\eta}{4\alpha} \cdot \left( \frac{2 \sqrt{\ell - 1} \phi}{7 \nu} - \frac{100\eta (\ell - 1)\zeta_2 L_1}{\alpha}  \right)\norm{\nabla f(\x_t)}\\
	&
	- \eta \cdot \frac{ \dotprod{\bS^\top \left(\nabla f(\x_t) - \nabla f(\x;\;\xi_t)\right), \bS^\top \nabla f(\x_t)} }{\sigma_t+\beta }\\
	&
	+ \frac{5\eta  \sqrt{\ell - 1}\alpha \hL^2\hzt_1}{8\nu}
	+ \frac{50\eta^2(\ell - 1) \zeta_2 L_0}{\alpha^2}\\
	=&
	f(\x_t) 
	- \frac{\eta}{4\alpha} \cdot \frac{\sqrt{\ell - 1}\phi}{7\nu} \cdot \norm{\nabla f(\x_t)}\\
	&
	- \eta \cdot \frac{ \dotprod{\bS^\top \left(\nabla f(\x_t) - \nabla f(\x;\;\xi_t)\right), \bS^\top \nabla f(\x_t)} }{\sigma_t+\beta }\\
	&
	+ \frac{5\eta  \sqrt{\ell - 1}\alpha \hL^2\hzt_1}{8\nu}
	+ \frac{25\eta^2(\ell - 1) \zeta_2 L_0}{\alpha^2},
\end{align*}
where the last inequality is because of setting $\eta \le \frac{\alpha \phi}{700\sqrt{\ell-1}\zeta_2 L_1}$.

Telescoping above equation, we can obtain that
\begin{align*}
	&\frac{\eta}{4\alpha} \cdot \frac{\sqrt{\ell - 1} \phi}{7\nu}\frac{1}{T} \sum_{t=0}^{T-1} \norm{\nabla f(\x_t)}\\
	\leq& 
	\frac{f(\x_0) - f(\x^*)}{T} 
	- \frac{\eta}{T} \cdot \sum_{t=0}^{T-1}\frac{ \dotprod{\bS^\top \left(\nabla f(\x_t) - \nabla f(\x;\;\xi_t)\right), \bS^\top \nabla f(\x_t)} }{\sigma_t+\beta }\\
	&
	+ \frac{5\eta  \sqrt{\ell - 1}\alpha \hL^2\hzt_1}{8\nu}
	+ \frac{25\eta^2(\ell - 1) \zeta_2 L_0}{\alpha^2}\\
	\stackrel{\eqref{eq:avr}}{\leq}&
	\frac{f(\x_0) - f(\x^*)}{T}
	+ \frac{\eta}{\sqrt{T}} \cdot \frac{9\sqrt{2} \sqrt{\ell - 1} \nu}{ \alpha} \sqrt{\log\frac{2}{\delta}}\\
	&
	+ \frac{5\eta  \sqrt{\ell - 1}\alpha \hL^2\hzt_1}{8\nu}
	+ \frac{25\eta^2(\ell - 1) \zeta_2 L_0}{\alpha^2}.
\end{align*}
Therefore,
\begin{align*}
	\frac{1}{T} \sum_{t=0}^{T-1} \norm{\nabla f(\x_t)}
	\leq&  
	\frac{28\nu \alpha \cdot (f(\x_0) - f(\x^*))}{\eta \sqrt{\ell-1} \phi T}
	+ \frac{252\sqrt{2}\nu^2}{ \phi \sqrt{T}}\sqrt{\log\frac{2}{\delta}} \\
	&
	+ \frac{35  \alpha^2 \hL^2\hzt_1}{2 \phi}
	+ \frac{700\eta \sqrt{(\ell - 1)} \zeta_2 L_0 \nu}{\alpha \phi}\\
	=&
	\frac{28\cdot 700\cdot \nu \sqrt{\zeta_2} L_1 \cdot (f(\x_0) - f(\x^*))}{\phi \sqrt{T}}
	+ \frac{252\sqrt{2}\nu^2}{ \phi \sqrt{T}}\sqrt{\log\frac{2}{\delta}} \\
	&
	+ \frac{35  \alpha^2 \hL^2\hzt_1}{2 \phi}
	+ \frac{  L_0 \nu \sqrt{\zeta_2}}{\phi\sqrt{T}L_1}\\
	=&
	\frac{28\cdot 700\cdot \nu \sqrt{\zeta_2} L_1 \cdot (f(\x_0) - f(\x^*))}{T^{1/4} }
	+ \frac{252\sqrt{2}\nu^2}{ T^{1/4}}\sqrt{\log\frac{2}{\delta}} \\
	&
	+ \frac{35  \alpha^2 \hL^2\hzt_1 T^{1/4}}{2 }
	+ \frac{  L_0 \nu \sqrt{\zeta_2}}{T^{1/4}L_1}\\
	\leq&
	\frac{28\cdot 700\cdot \nu \sqrt{\zeta_2} L_1 \cdot (f(\x_0) - f(\x^*))}{T^{1/4} }
	+ \frac{252\sqrt{2}\nu^2}{ T^{1/4}}\sqrt{\log\frac{2}{\delta}} \\
	&
	+ \frac{35   \hL^2\hzt_1 }{2 T^{1/4} }
	+ \frac{  L_0 \nu \sqrt{\zeta_2}}{T^{1/4}L_1},
\end{align*}
where the first equality is because of $\eta = \frac{\alpha }{ 700 \sqrt{T} \cdot \sqrt{\ell - 1} \sqrt{\zeta_2} L_1}$, the second equality is because of $\phi = \frac{1}{T^{1/4}}$, and  the last inequality is because of $\alpha \leq T^{-1/2}$.

Since it holds that
\begin{align*}
	\min_{t=0,\dots, T-1} \norm{\nabla f(\x_t)} 
	\leq 
	\frac{1}{T} \sum_{t=0}^{T-1} \norm{\nabla f(\x_t)},
\end{align*}
then it holds that 
\begin{align*}
	\min_{t=0,\dots, T-1} \norm{\nabla f(\x_t)}
	\leq 
\frac{28\cdot 700\cdot \nu \sqrt{\zeta_2} L_1 \cdot (f(\x_0) - f(\x^*))}{T^{1/4} }
+ \frac{252\sqrt{2}\nu^2}{ T^{1/4}}\sqrt{\log\frac{2}{\delta}} 
+ \frac{35   \hL^2\hzt_1 }{2 T^{1/4} }
+ \frac{  L_0 \nu \sqrt{\zeta_2}}{T^{1/4}L_1}.
\end{align*}

Next, we will prove the result in Eq.~\eqref{eq:P2}.
By a union bound and the definition of $\hE_t$, we have
\begin{align*}
	\PP\left(\cE_t \right) 
	\geq 
	1 -  \left(\sum_{i=1}^4 \PP\left(\overline{\hE_{t,i}}\right) + \sum_{i=2}^{3}\PP\left(\overline{\cE_{t,i}}\right)\right)
	\geq 
	1 - \left( \delta + \delta + \delta + \exp(-\cO\left(\ell\right)) + \delta + \delta \right)
	\geq
	1 - 6\delta,
\end{align*}
where $\overline{\cE_{t,i}}$ denotes the complement of $\cE_{t,i}$, the second inequality is because of Lemma~\ref{lem:E11}-Lemma~\ref{lem:E41} and  Lemma~\ref{lem:E2}-Lemma~\ref{lem:E3}, and the last inequality is because of $\ell = \cO\left( k\log\frac{1}{\delta} \right)$.

Using the union bound again, we can obtain that
\begin{align*}
	\PP\left(\cap_{t=0}^{T-1} \;\hE_t\right) 
	\geq 
	1 - \sum_{t=0}^{T-1} \PP\left(\overline{\hE_t}\right)
	\geq
	1 - 6 T\delta,
\end{align*}
which concludes the proof.
\end{proof}

\begin{corollary}\label{cor:main1}
Let functions $f(\x)$ and $f(\x;\;\xi)$ satisfy the properties described in Theorem~\ref{thm:main1}.
Let $\bS_t\in\RR^{d\times \ell}$ be the Gaussian or Rademacher sketching matrices with $\ell = 16k\log({4T}/{\delta})$ and $\delta\in(0,1)$.
Set the smooth parameter $\alpha$ and step size $\eta$ satisfy
\begin{align*}
	\alpha^2 \leq \min\left\{ \frac{1}{2L_1^2\zeta_1 T},\; \frac{3\varepsilon^2}{8L_0^2\zeta_1}\right\} 
	\qquad
	\mbox{and}
	\qquad
	\eta = \frac{\alpha}{\sqrt{(\ell - 1)\zeta_2 T} },
\end{align*}
with $T>0$ being large enough that $\eta \leq {3\alpha}/{(64 L_1 \zeta_2 \sqrt{\ell - 1})}$ and with $0<\varepsilon$ being the target precision. 
Then if the total iteration number $T$ satisfies 
\begin{equation}\label{eq:T1}
	T = \cO\left(\max\left\{
	\frac{1}{\varepsilon^4},
	\frac{ \norm{\bH}^2 \nu^4}{\varepsilon^4},
	\frac{\Big(\tr(\bH)\Big)^2}{k^2 \varepsilon^4}
	\right\}
	\right),
\end{equation}
then, with a probability at least $1-\delta$, the sequence $\{\x_t\}$ with $t = 0,\dots, T-1$ generated by Algorithm~\ref{alg:SA} satisfies that
\begin{align*}
	\min_{t=0,\dots, T-1} \norm{\nabla f(\x_t)} \leq \varepsilon.
\end{align*}
Furthermore, the total query complexity is 
\begin{equation}\label{eq:Q1}
	Q = 	\cO\left(\max\left\{
	\frac{k}{\varepsilon^4},
	\frac{ k\norm{\bH}^2 \nu^4 }{\varepsilon^4},
	\frac{\Big(\tr(\bH)\Big)^2 \nu^4}{k \varepsilon^4}
	\right\}\cdot \log\frac{1}{\delta\varepsilon}
	\right).
\end{equation}
\end{corollary}
\begin{proof}
First, we will show that $\x_t \in \cK$	defined in Assumption~\ref{ass:hL}.
By the definition of $\x_t$, we can obtain that
\begin{align*}
	\norm{\x_t - \x_0}_{\bH} 
	=
	\norm{\x_{t-1} - \eta \frac{\hg(\x_{t-1})}{\sigma_{t-1}} - \x_0}_{\bH}
	\leq 
	\norm{ \x_{t-1} - \x_0 }_{\bH} 
	+ 
	\eta \norm{\frac{\hg(\x_{t-1})}{\sigma_{t-1}}}_{\bH}
	\leq
	\eta \sum_{i=0}^{t-1} \norm{\frac{\hg(\x_{t-1})}{\sigma_{t-1}}}_{\bH}.
\end{align*}
Combining with Eq.~\eqref{eq:up} and the setting of step size, we can obtain that
\begin{align*}
	\norm{\x_t - \x_0}_{\bH} 
	\leq& 
	\eta \sum_{i=0}^{t-1} \left(\frac{2\sqrt{6\zeta_2(\ell-1)}}{\alpha} + \frac{2 \hL \sqrt{(\ell-1)\zeta_2\hzt_1}}{\nu}\right)\\
	\leq&
	T \eta  \left(\frac{2\sqrt{6\zeta_2(\ell-1)}}{\alpha} + \frac{2 \hL \sqrt{(\ell-1)\zeta_2\hzt_1}}{\nu}\right)\\
	\leq&
	T \cdot \frac{\alpha }{ 700 \sqrt{T} \cdot \sqrt{\ell - 1} \zeta_2 L_1} \left(\frac{2\sqrt{6\zeta_2(\ell-1)}}{\alpha} + \frac{2 \hL \sqrt{(\ell-1)\zeta_2\hzt_1}}{\nu}\right)\\
	=&
	\frac{\sqrt{6 T}}{350 L_1 \sqrt{\zeta_2}}
	+ 
	\frac{\hL \alpha\sqrt{\hzt_1 T}}{350\sqrt{\zeta_2} L_1 \nu}\\
	\leq&
	\frac{\sqrt{6 T}}{350 L_1 \sqrt{\zeta_2}} 
	+ 
	\frac{\sqrt{T}}{350 \sqrt{8} L_1 \sqrt{\zeta_2}}\\
	\leq& 
	\frac{3\sqrt{ T}}{350 L_1 \sqrt{\zeta_2}},
\end{align*}
where the fifth inequality is because of $\alpha = \min\{\frac{1}{T^{1/2}}, \sqrt{ \frac{\nu^2}{8\hL^2 \hzt_1} }\}$.
Thus, we obtain that all $\x_t$ belongs to the set $\cK$.

Next, we will prove that $\x_t + \alpha\si$ also belongs to the set $\cK$. 
\begin{align*}
	\norm{\x_t + \alpha \si - \x_0}_{\bH}
	\leq 
	\norm{\x_t - \x_0}_{\bH} + \alpha \norm{\si}_{\bH} 
	\stackrel{\eqref{eq:tr}}{\leq}
	\norm{\x_t - \x_0}_{\bH} + 2\alpha \sqrt{\tr(\bH)}
	\leq
	\frac{3\sqrt{ T}}{350 L_1 \sqrt{\zeta_2}} 
	+  2\alpha \sqrt{\tr(\bH)}.
\end{align*}
Combining with the setting of $\alpha \leq $, we can obtain that  $\x_t + \alpha\si$ is in $\cK$.

By setting the right hand side of Eq.~\eqref{eq:main1} to $\varepsilon$, then we only require $T$ satisfy that
\begin{align*}
	T 
	= \cO\left(\max\left\{
	\frac{1}{\varepsilon^4},
	\frac{\zeta_2^2 \nu^4}{\varepsilon^4}
	\right\}
	\right)
	=
	\cO\left(\max\left\{
	\frac{1}{\varepsilon^4},
	\frac{ \norm{\bH}^2 \nu^4}{\varepsilon^4},
	\frac{\Big(\tr(\bH)\Big)^2}{k^2 \varepsilon^4}
	\right\}
	\right),
\end{align*}
where the last equality is because of the definition of $\zeta_2$ in Eq.~\eqref{eq:zeta_2}.

Then, the total query complexity is 
\begin{align*}
	Q 
	= 
	T\ell
	=& 
	\cO\left(\max\left\{
	\frac{k}{\varepsilon^4},
	\frac{ k\norm{\bH}^2 \nu^4 }{\varepsilon^4},
	\frac{\Big(\tr(\bH)\Big)^2 \nu^4}{k \varepsilon^4}
	\right\}\cdot \log\frac{T}{\delta}
	\right)\\
	=&
	\cO\left(\max\left\{
	\frac{k}{\varepsilon^4},
	\frac{ k\norm{\bH}^2 \nu^4 }{\varepsilon^4},
	\frac{\Big(\tr(\bH)\Big)^2 \nu^4}{k \varepsilon^4}
	\right\}\cdot \log\frac{1}{\delta\varepsilon}
	\right).
\end{align*}
\end{proof}

\begin{remark}
	Theorem~\ref{thm:main1} and Corollary~\ref{cor:main1} show Algorithm~\ref{alg:SA1} can achieve an 
    \begin{align*}
    \cO\left(\max\left\{
	\frac{1}{\varepsilon^4},
	\frac{ \norm{\bH}^2 \nu^4}{\varepsilon^4},
	\frac{\Big(\tr(\bH)\Big)^2}{k^2 \varepsilon^4}
	\right\}
	\right)    
    \end{align*}
    iteration complexity. 
	If we increase $k$ which means increasing the batch size $\ell$, Algorithm~\ref{alg:SA1} can find an $\varepsilon$-stationary point with less steps.
	At the same time, Eq.~\eqref{eq:Q1} shows that  Algorithm~\ref{alg:SA1} requires less queries if we increase the sample size $\ell$ until $\ell = 16k\log\frac{T}{\delta}$ with $k = \frac{\tr(\bH)}{\norm{\bH}}$.
	This is different from that of Algorithm~\ref{alg:SA} whose query complexity  can \emph{not} benefit from increasing the batch size just as discussed in Remark~\ref{rmk:main}.
\end{remark}
\begin{remark}\label{rmk:Q1}
	If we choose batch size $\ell = 16k\log({T}/{\delta})$ with $k = {\tr(\bH)}/{\norm{\bH}}$, then Eq.~\eqref{eq:Q1} will reduce to 
	\begin{align*}
		Q 
		=& 
		\cO\left(\max\left\{
		\frac{\tr(\bH)}{\norm{\bH} \varepsilon^4},
		\frac{ \norm{\bH} \cdot \tr(\bH) \cdot \nu^4 }{\varepsilon^4}
		\right\}\cdot \log\frac{1}{\delta\varepsilon}
		\right)\\
		=&
		\cO\left(\max\left\{
		\frac{\tr(\bH)}{\norm{\bH} \varepsilon^4},
		\frac{\tr(\bH)}{\norm{\bH}} \cdot \frac{ \norm{\bH}^2  \cdot \nu^4 }{\varepsilon^4}
		\right\}\cdot \log\frac{1}{\delta\varepsilon}
		\right).
	\end{align*}
	Thus, the query complexity of Algorithm~\ref{alg:SA1} also achieves a weak dimension dependency just as the Algorithm~\ref{alg:SA} does. 
	In many real applications, we can choose proper matrix $\bH$ such that $\frac{\tr(\bH)}{\norm{\bH}} \ll d$.
\end{remark}

{
\begin{remark}\label{rmk:sto_M}
Similar to the statement of Remark \ref{remark:deter-M}, even by taking $\hg(\x_t)=\nabla f(\x_t;\xi_t)$ and Assumption~\ref{ass:LL} with $\bH=\bI$, finding an $\varepsilon$-stationary point via the fixed stepsize update
$\x_{t+1}=\x_t - \eta\hg(\x_t)$
with $\eta_t>0$ 
requires the iteration numbers that depends on 
$M$ \citep[Theorem 8]{zhang2019gradient}.
In contrast, our Corollary \ref{cor:main1} shows that
the iteration numbers required by Algorithm~\ref{alg:SA1} only depends on $\norm{\bH}$ and $\tr(\bH)$.
\end{remark}}

\section{Conclusion}

In this work, we establish a theoretical explanation for the effectiveness of adaptive zeroth-order optimization methods. By characterizing the relationship between function value variability and gradient magnitude, we show that the empirical standard deviation of sampled function values serves as a reliable proxy for the norm of the (stochastic) gradient with high probability. This result provides a rigorous justification for standard deviation normalization as a principled adaptive step-size control mechanism in ZO optimization.

Building on this insight, we analyze adaptive ZO methods under a generalized $(L_0,L_1)$-smoothness condition defined with respect to a positive semi-definite matrix norm. This setting goes beyond classical smoothness assumptions and captures anisotropic geometries commonly observed in modern deep neural networks, including large language models. Our analysis yields explicit convergence rates and query complexity bounds for both deterministic and stochastic objectives.

A central implication of our results is that adaptive ZO methods can achieve strictly faster convergence and significantly reduced query complexity compared to fixed-step vanilla ZO methods, by automatically adjusting to the local geometry of the optimization landscape. These findings bridge the gap between empirical practice and theoretical understanding of adaptive ZO algorithms. Looking forward, this framework opens avenues for analyzing more sophisticated adaptive mechanisms and for extending ZO optimization theory to broader classes of non-smooth and large-scale learning problems.

\pb
\clearpage
\bibliography{ref.bib}

\begin{thebibliography}{}

\bibitem[Borodich \& Kovalev, 2025]{borodich2025nesterov}
Borodich, E. \& Kovalev, D. (2025).
\newblock Nesterov finds graal: Optimal and adaptive gradient method for convex
  optimization.
\newblock {\em arXiv preprint arXiv:2507.09823}.

\bibitem[Chen et~al., 2017]{chen2017zoo}
Chen, P.-Y., Zhang, H., Sharma, Y., Yi, J., \& Hsieh, C.-J. (2017).
\newblock Zoo: Zeroth order optimization based black-box attacks to deep neural
  networks without training substitute models.
\newblock In {\em Proceedings of the 10th ACM Workshop on Artificial
  Intelligence and Security}  (pp.\ 15--26).: ACM.

\bibitem[Chen et~al., 2023]{chen2023generalized}
Chen, Z., Zhou, Y., Liang, Y., \& Lu, Z. (2023).
\newblock Generalized-smooth nonconvex optimization is as efficient as smooth
  nonconvex optimization.
\newblock In {\em International Conference on Machine Learning}  (pp.\
  5396--5427).

\bibitem[Chezhegov et~al., 2025]{chezhegov2025convergence}
Chezhegov, S., Beznosikov, A., Horv{\'a}th, S., \& Gorbunov, E. (2025).
\newblock Convergence of clipped-{SGD} for convex {$(L_0, L_1)$}-smooth
  optimization with heavy-tailed noise.
\newblock {\em arXiv preprint arXiv:2505.20817}.

\bibitem[Chrabaszcz et~al., 2018]{chrabaszcz2018back}
Chrabaszcz, P., Loshchilov, I., \& Hutter, F. (2018).
\newblock Back to basics: Benchmarking canonical evolution strategies for
  playing atari.
\newblock {\em arXiv preprint arXiv:1802.08842}.

\bibitem[Cohen, 2016]{cohen2016simpler}
Cohen, M.~B. (2016).
\newblock Simpler and tighter analysis of sparse oblivious subspace embeddings.
\newblock In {\em Proceedings of the 27th Annual ACM-SIAM Symposium on Discrete
  Algorithms (SODA)}.

\bibitem[Cohen et~al., 2016]{cohen2016optimal}
Cohen, M.~B., Nelson, J., \& Woodruff, D.~P. (2016).
\newblock Optimal approximate matrix product in terms of stable rank.
\newblock In {\em International Colloquium on Automata, Languages, and
  Programming}: Schloss Dagstuhl-Leibniz-Zentrum fur Informatik GmbH, Dagstuhl
  Publishing.

\bibitem[Cooper, 2022]{cooper2024empirical}
Cooper, Y. (2022).
\newblock An empirical study of the {$(L_0, L_1)$}-smoothness condition.
\newblock In {\em Workshop on Mathematics of Modern Machine Learning}.

\bibitem[Cortinovis \& Kressner, 2021]{cortinovis2021randomized}
Cortinovis, A. \& Kressner, D. (2021).
\newblock On randomized trace estimates for indefinite matrices with an
  application to determinants.
\newblock {\em Foundations of Computational Mathematics}, (pp.\ 1--29).

\bibitem[Crawshaw \& Liu, 2025]{crawshaw2025complexity}
Crawshaw, M. \& Liu, M. (2025).
\newblock Complexity lower bounds of adaptive gradient algorithms for
  non-convex stochastic optimization under relaxed smoothness.
\newblock {\em arXiv preprint arXiv:2505.04599}.

\bibitem[Crawshaw et~al., 2022]{crawshaw2022robustness}
Crawshaw, M., Liu, M., Orabona, F., Zhang, W., \& Zhuang, Z. (2022).
\newblock Robustness to unbounded smoothness of generalized {SignSGD}.
\newblock In {\em Advances in Neural Information Processing Systems}  (pp.\
  9955--9968).

\bibitem[Dang et~al., 2025]{dang2025fzoo}
Dang, S., Guo, Y., Zhao, Y., Ye, H., Zheng, X., Dai, G., \& Tsang, I. (2025).
\newblock Fzoo: Fast zeroth-order optimizer for fine-tuning large language
  models towards adam-scale speed.
\newblock {\em arXiv preprint arXiv:2506.09034}.

\bibitem[Gaash et~al., 2025]{gaash2025convergence}
Gaash, O., Levy, K.~Y., \& Carmon, Y. (2025).
\newblock Convergence of clipped {SGD} on convex {$(L_0, L_1)$}-smooth
  functions.
\newblock {\em arXiv preprint arXiv:2502.16492}.

\bibitem[Ghadimi \& Lan, 2013]{ghadimi2013stochastic}
Ghadimi, S. \& Lan, G. (2013).
\newblock Stochastic first-and zeroth-order methods for nonconvex stochastic
  programming.
\newblock {\em SIAM Journal on Optimization}, 23(4), 2341--2368.

\bibitem[Gorbunov et~al., 2024]{gorbunov2024methods}
Gorbunov, E., Tupitsa, N., Choudhury, S., Aliev, A., Richt{\'a}rik, P.,
  Horv{\'a}th, S., \& Tak{\'a}{\v{c}}, M. (2024).
\newblock Methods for convex {$(L_0, L_1)$}-smooth optimization: Clipping,
  acceleration, and adaptivity.
\newblock {\em arXiv preprint arXiv:2409.14989}.

\bibitem[Ilyas et~al., 2018]{ilyas18a}
Ilyas, A., Engstrom, L., Athalye, A., \& Lin, J. (2018).
\newblock Black-box adversarial attacks with limited queries and information.
\newblock In J. Dy \& A. Krause (Eds.), {\em Proceedings of the 35th
  International Conference on Machine Learning}, volume~80 of {\em Proceedings
  of Machine Learning Research}  (pp.\ 2137--2146).  Stockholmsmässan,
  Stockholm Sweden: PMLR.

\bibitem[Jiang et~al., 2025]{jiang2025decentralized}
Jiang, Z., Balu, A., \& Sarkar, S. (2025).
\newblock Decentralized relaxed smooth optimization with gradient descent
  methods.
\newblock {\em arXiv preprint arXiv:2508.08413}.

\bibitem[Khirirat et~al., 2024a]{khirirat2024communication}
Khirirat, S., Sadiev, A., Riabinin, A., Gorbunov, E., \& Richt{\'a}rik, P.
  (2024a).
\newblock Communication-efficient algorithms under generalized smoothness
  assumptions.
\newblock {\em OpenReview}.

\bibitem[Khirirat et~al., 2024b]{khirirat2024error}
Khirirat, S., Sadiev, A., Riabinin, A., Gorbunov, E., \& Richt{\'a}rik, P.
  (2024b).
\newblock Error feedback under {$(L_0, L_1)$}-smoothness: Normalization and
  momentum.
\newblock {\em arXiv preprint arXiv:2410.16871}.

\bibitem[Koloskova et~al., 2023]{koloskova2023revisiting}
Koloskova, A., Hendrikx, H., \& Stich, S.~U. (2023).
\newblock Revisiting gradient clipping: Stochastic bias and tight convergence
  guarantees.
\newblock In {\em International Conference on Machine Learning}  (pp.\
  17343--17363).

\bibitem[Laurent \& Massart, 2000]{Laurent2000Adaptive}
Laurent, B. \& Massart, P. (2000).
\newblock Adaptive estimation of a quadratic functional by model selection.
\newblock {\em Annals of Statistics}, 28(5), 1302--1338.

\bibitem[Li et~al., 2023]{li2023convergence}
Li, H., Rakhlin, A., \& Jadbabaie, A. (2023).
\newblock Convergence of adam under relaxed assumptions.
\newblock {\em Advances in Neural Information Processing Systems}, (pp.\
  52166--52196).

\bibitem[Li et~al., 2024]{li2024problem}
Li, J., Chen, X., Ma, S., \& Hong, M. (2024).
\newblock Problem-parameter-free decentralized nonconvex stochastic
  optimization.
\newblock {\em arXiv preprint arXiv:2402.08821}.

\bibitem[Lobanov \& Gasnikov, 2025]{lobanov2025power}
Lobanov, A. \& Gasnikov, A. (2025).
\newblock Power of generalized smoothness in stochastic convex optimization:
  First-and zero-order algorithms.
\newblock {\em arXiv preprint arXiv:2501.18198}.

\bibitem[Lobanov et~al., 2024]{lobanov2024linear}
Lobanov, A., Gasnikov, A., Gorbunov, E., \& Tak{\'a}{\v{c}}, M. (2024).
\newblock Linear convergence rate in convex setup is possible! gradient descent
  method variants under {$(L_0, L_1)$}-smoothness.
\newblock {\em arXiv preprint arXiv:2412.17050}.

\bibitem[Lyu \& Tsang, 2019]{lyu2019black}
Lyu, Y. \& Tsang, I.~W. (2019).
\newblock Black-box optimizer with implicit natural gradient.
\newblock {\em arXiv preprint arXiv:1910.04301}.

\bibitem[Malladi et~al., 2023]{malladi2023fine}
Malladi, S., Gao, T., Nichani, E., Damian, A., Lee, J.~D., Chen, D., \& Arora,
  S. (2023).
\newblock Fine-tuning language models with just forward passes.
\newblock {\em Advances in Neural Information Processing Systems}, 36,
  53038--53075.

\bibitem[Mania et~al., 2018]{mania2018simple}
Mania, H., Guy, A., \& Recht, B. (2018).
\newblock Simple random search of static linear policies is competitive for
  reinforcement learning.
\newblock In {\em Proceedings of the 32nd International Conference on Neural
  Information Processing Systems}  (pp.\ 1805--1814).

\bibitem[Nesterov \& Spokoiny, 2017]{Nesterov2017}
Nesterov, Y. \& Spokoiny, V. (2017).
\newblock Random gradient-free minimization of convex functions.
\newblock {\em Foundations of Computational Mathematics}, 17(2), 527--566.

\bibitem[Qiu et~al., 2025]{qiu2025evolution}
Qiu, X., Gan, Y., Hayes, C.~F., Liang, Q., Meyerson, E., Hodjat, B., \&
  Miikkulainen, R. (2025).
\newblock Evolution strategies at scale: Llm fine-tuning beyond reinforcement
  learning.
\newblock {\em arXiv preprint arXiv:2509.24372}.

\bibitem[Reisizadeh et~al., 2025]{reisizadeh2025variance}
Reisizadeh, A., Li, H., Das, S., \& Jadbabaie, A. (2025).
\newblock Variance-reduced clipping for non-convex optimization.
\newblock In {\em International Conference on Acoustics, Speech and Signal
  Processing}  (pp.\ 1--5).

\bibitem[Salimans et~al., 2017]{salimans2017evolution}
Salimans, T., Ho, J., Chen, X., Sidor, S., \& Sutskever, I. (2017).
\newblock Evolution strategies as a scalable alternative to reinforcement
  learning.
\newblock {\em arXiv preprint arXiv:1703.03864}.

\bibitem[Tovmasyan et~al., 2025]{tovmasyan2025revisiting}
Tovmasyan, Z., Malinovsky, G., Condat, L., \& Richt{\'a}rik, P. (2025).
\newblock Revisiting stochastic proximal point methods: Generalized smoothness
  and similarity.
\newblock {\em arXiv preprint arXiv:2502.03401}.

\bibitem[Tyurin, 2025]{tyurin2024toward}
Tyurin, A. (2025).
\newblock Toward a unified theory of gradient descent under generalized
  smoothness.
\newblock In {\em International Conference on Learning Representations}.

\bibitem[Vankov et~al., 2025]{vankov2024optimizing}
Vankov, D., Rodomanov, A., Nedich, A., Sankar, L., \& Stich, S.~U. (2025).
\newblock Optimizing {$(L_0, L_1)$}-smooth functions by gradient methods.
\newblock In {\em International Conference on Learning Representations}.

\bibitem[Woodruff et~al., 2014]{woodruff2014sketching}
Woodruff, D.~P. et~al. (2014).
\newblock Sketching as a tool for numerical linear algebra.
\newblock {\em Foundations and Trends{\textregistered} in Theoretical Computer
  Science}, 10(1--2), 1--157.

\bibitem[Xie et~al., 2024]{xie2024trust}
Xie, C., Li, C., Zhang, C., Deng, Q., Ge, D., \& Ye, Y. (2024).
\newblock Trust region methods for nonconvex stochastic optimization beyond
  lipschitz smoothness.
\newblock In {\em AAAI Conference on Artificial Intelligence}  (pp.\
  16049--16057).

\bibitem[Ye et~al., 2025]{ye2025unified}
Ye, H., Chang, X., \& Chen, X. (2025).
\newblock A unified zeroth-order optimization framework via oblivious
  randomized sketching.
\newblock {\em arXiv preprint arXiv:2510.10945}.

\bibitem[Yu et~al., 2025]{yu2025mirror}
Yu, D., Jiang, W., Wan, Y., \& Zhang, L. (2025).
\newblock Mirror descent under generalized smoothness.
\newblock {\em arXiv preprint arXiv:2502.00753}.

\bibitem[Zhang et~al., 2020a]{zhang2020improved}
Zhang, B., Jin, J., Fang, C., \& Wang, L. (2020a).
\newblock Improved analysis of clipping algorithms for non-convex optimization.
\newblock In {\em Advances in Neural Information Processing Systems}  (pp.\
  15511--15521).

\bibitem[Zhang et~al., 2020b]{zhang2019gradient}
Zhang, J., He, T., Sra, S., \& Jadbabaie, A. (2020b).
\newblock Why gradient clipping accelerates training: A theoretical
  justification for adaptivity.
\newblock {\em International Conference on Learning Representations}.

\end{thebibliography}
\bibliographystyle{apalike2}

\appendix

\section{Useful Lemmas}

\begin{lemma}[Azuma-Hoeffding Inequality]\label{lem:concent}
	Let $E_1,\dots,E_T$ be a martingale difference sequence with a uniform bound $|E_i|\leq B$ for all $i$.
	Then,
	\begin{equation}
		\Pr\left( \frac{1}{s} \sum_{t=1}^{s} E_t > \zeta \right) \leq \exp\left( -\frac{\zeta^2 }{2 B} \cdot s \right).
	\end{equation}
\end{lemma}

\begin{lemma}\label{lem:gss}
If $\bS$ is the Gaussian sketching matrix, that is, its entries  are independently identically distribution drawn from $\cN(0, 1/\ell)$ with $\ell \geq 4$, then with a probability at least $1 - \delta'$ with 
\begin{align*}
	\delta'= \exp\left(- \left( \frac{\ell}{8} - \sqrt{\frac{\ell - 2}{8}} \right)\right),
\end{align*}
such that $  \frac{1}{\ell}  \left( \sum_{i=1}^{\ell} \dotprod{\nabla f(\x), \si} \right)^2 \leq \frac{1}{4}\norm{\nabla f(\x)}^2$.
\end{lemma}
\begin{proof}
	First, if $S$ is the Gaussian sketching matrix, that is, its entries  are independently identically distribution drawn from $\cN(0, 1/\ell)$. 
	Then by the property of Gaussian distribution,  we can obtain that $\dotprod{\nabla f(\x), \si} \sim \cN(0, \norm{\nabla f(\x)}^2/\ell)$.
	Since $\si$'s are independent, we can obtain that 
\begin{align*}
 \sum_{i=1}^{\ell} \dotprod{\nabla f(\x), \si} \sim \cN(0, \norm{\nabla f(\x)}^2).
\end{align*}
Thus, $(\sum_{i=1}^{\ell} \dotprod{\nabla f(\x), \si})^2$ follows a $\chi^2(1)$ distribution scaled with $\norm{\nabla f(\x)}^2$.

Then, by Lemma~\ref{lem:chi}, the probability that $ \frac{1}{\ell} (\sum_{i=1}^{\ell} \dotprod{\nabla f(\x), \si})^2$ larger than $1/4 \norm{\nabla f(\x)}^2$ is no larger than 

\begin{align*}
\PP\left[q_1 \geq 1 + 2 \sqrt{ \tau} + 2 \tau = \frac{\ell}{4} \right] \leq \exp(-\tau),
\end{align*}
with $\sqrt{\tau} = \sqrt{\frac{\ell - 2}{8}} - \frac{1}{2}$ which implies $\tau = \frac{\ell}{8} - \sqrt{\frac{\ell - 2}{8}}$.
\end{proof}

\begin{lemma}[$\chi^2$ tail bound~\cite{Laurent2000Adaptive}]
	\label{lem:chi}
	Let $q_1, \dots, q_n$ be independent $\chi^2$ random variables, each with one degree of freedom. For any vector $\gamma = (\gamma_1, \dots , \gamma_n) \in \RR_+^n$ with non-negative
	entries, and any $\tau > 0$,
	\[
	\PP\left[\sum_{i=1}^{n}\gamma_iq_i \geq \norm{\gamma}_1 + 2 \sqrt{\norm{\gamma}_2^2 \tau} + 2\norm{\gamma}_\infty \tau\right] \leq \exp(-\tau),
	\]
	where $\norm{\gamma}_1 = \sum_{i=1}^{n} |\gamma_i|$.
\end{lemma} 

\begin{lemma}[Lemma 11 of \citet{ye2025unified}]\label{lem:bound_S_F_norm}
	If $\bS\in\RR^{d\times \ell}$ is an oblivious  $(\frac{1}{4}, k, \delta)$-random sketching matrix, then with a probability at least $1-\delta$, it holds that
	\begin{equation}\label{eq:S_norm}
		\norm{\bS}_F^2 \leq \frac{5\ell}{4} + \frac{d \ell}{4k}.
	\end{equation}
\end{lemma}

\begin{lemma}\label{lem:HW}
	Let $\bm{r} \in\RR^d $ be a vector of mean $0$,i.i.d. sub-Gaussian random
	variables with constant sub-Gaussian parameter $C$. Let $\bA\in\RR^{d\times d}$ be a matrix. Then, there exists
	a constant $c$ only depending on $C$ such that for every $\tau\geq 0$,
	\begin{equation}
		\PP\left\{ \left| \bm{r}^\top \bA \bm{r} - \EE\left[\bm{r}^\top \bA \bm{r}\right] \right| > \tau\right\} \leq 2 \exp\left(-c \cdot \min\left\{ \frac{\tau^2}{\norm{\bA}_F^2},\; \frac{\tau}{\norm{\bA}_2} \right\}\right)
	\end{equation}
\end{lemma}

\begin{lemma}\label{lem:ss}
	If $\bS\in\RR^{d\times \ell}$ is random sketching matrix which is Gaussian, Rademacher, SRHT, or Sparse Embedding (refer to Sec.~\ref{subsec:Sketch}), then with  a probability at least $1 - \delta'$ with $\delta' = \exp\left(-\cO(\ell)\right)$, it holds that  $ \frac{1}{\ell}  \left( \sum_{i=1}^{\ell} \dotprod{\nabla f(\x), \si} \right)^2 \leq \frac{1}{4}\norm{\nabla f(\x)}^2$.
\end{lemma}
\begin{proof}
First, we 
	\begin{align*}
		\left(	\sum_{i=1}^{\ell} \dotprod{\nabla f(\x), \si}\right)^2 = \nabla^\top f(\x) \bS\bm{1}_\ell \bm{1}_\ell^\top \bS^\top \nabla f(\x) = \left(\bS\bm{1}_\ell\right)^\top \nabla f(\x) \nabla^\top f(\x) \bS\bm{1}_\ell.
	\end{align*}
	Our choice of sketching matrix $\bS$ guarantees that $\bS\bm{1}_\ell$ is a vector of mean $0$, i.i.d. sub-Gaussian random
	variables. 
	
	Furthermore, it holds that 
	\begin{align*}
		\EE \left[\left(	\sum_{i=1}^{\ell} \dotprod{\nabla f(\x), \si}\right)^2\right]  
		=
		\sum_{i=1}^{\ell} \EE\left[ \dotprod{\nabla f(\x), \si}^2 \right]
		=
		\sum_{i=1}^{\ell} \EE \left[ \nabla^\top f(\x) \si \si^\top \nabla f(\x) \right]
		= \norm{ \nabla f(\x) }^2,
	\end{align*}
	where the first equality is because of $\si$'s are independent to each other.

	By Lemma~\ref{lem:HW} with $\bm{r} = \bS\bm{1}_\ell$, $\bA = \nabla f(\x) \nabla^\top f(\x)$, and $\tau = \left(\frac{\ell}{4} - 1\right)\norm{\nabla f(\x)}^2 $, to achieve  $   \frac{1}{\ell}  \left( \sum_{i=1}^{\ell} \dotprod{\nabla f(\x), \si} \right)^2 \leq \frac{1}{4}\norm{\nabla f(\x)}^2$, it holds with a probability at least $1 - \delta'$ with $\delta' = \exp\left(-\cO(\ell)\right)$.
\end{proof}

\begin{lemma}[Corollary 2 of \cite{cortinovis2021randomized}]\label{lem:trace}
	Let $\bA \in \RR^{d\times d}$ be positive definite, $\delta \in (0, 1/2]$, $\varepsilon\in(0,1)$, $\ell \in \mathbb{N}$. Let $\tau_\ell(\bA) = \sum_{i=1}^{\ell} \si^\top \bA \si$. To achieve the following property  holding with a probability of at least $1-\delta$,
	\begin{equation*}
		\Big|\tau_\ell(\bA) - \tr(\bA)\Big| \le \varepsilon \cdot \tr(\bA),
	\end{equation*}
	then, 
	\begin{enumerate}
		\item if $\si$ is the $i$-th column of a random  Gaussian sketching matrix, $\ell = 8\varepsilon^{-2}\frac{\norm{A}_2}{\tr(A)}\log\frac{2}{\delta}$ is sufficient.
		\item  if $\si \in\RR^d$ is the $i$-th column of a random Rademacher matrix, $\ell = 8\varepsilon^{-2}(1+\varepsilon)\frac{\norm{A}_2}{\tr(A)}\log\frac{2}{\delta}$ is sufficient.
		\item if $\si \in\RR^d$ is a column of SRHT, $\ell = \left(1 + \sqrt{8\ln\frac{2n}{\delta}}\right)^2\varepsilon^{-2}\log\frac{2}{\delta}$ is sufficient.
	\end{enumerate}
\end{lemma}

\begin{lemma}\label{lem:tr}
Let $\bS_t\in\RR^{d\times \ell}$ be the Gaussian or Rademacher sketching matrices with $\ell = 16k\log\frac{4T}{\delta}$ and $0<\delta<1$. 
Then, for all $\si$ which is a column of $\bS_t$, with a probability at least $1-\delta$, it holds that
\begin{equation}\label{eq:tr}
\norm{\si}_{\bH}
\leq 
2 \sqrt{\tr(\bH)}
\end{equation}
\end{lemma}
\begin{proof}
By the properties of sketching matrices and the value of $\ell$, we can obtain that
\begin{align*}
	\norm{\si}_{\bH}
	= \sqrt{ \norm{\si}_{\bH}^2 }
	\leq 
	\sqrt{\sum_{i=1}^{\ell} \norm{\si}_{\bH}^2}
	=
	\sqrt{\sum_{i=1}^{\ell} \si^\top \bH \si}
	\leq
	2\sqrt{\tr(\bH)},
\end{align*}
where the last inequality is because of Lemma~\ref{lem:trace} and the value of $\ell$.
\end{proof}

{
\section{The Connection between Assumption \ref{ass:LL} and Related Work}

Recall that the original relaxed smoothness condition \citep{zhang2019gradient} says the differentiable function $f:\BR^d\to\BR$ holds
\begin{align}\label{eq:second-L0-L1}
    \norm{\nabla^2 f(\x)} \leq L_0 + L_1\norm{\nabla f(\x)}
\end{align}
for all $\vx\in\BR^d$.
Later \citet[Corollary A.3 and A.4]{zhang2020improved} proved (\ref{eq:second-L0-L1}) implies it holds
\begin{align}\label{eq:first-L0-L1-1}
    \norm{\nabla f(\x)-\nabla f(\y)} \leq (AL_0 + BL_1\norm{\nabla f(\x)})\norm{\x-\y}.
\end{align}
and
\begin{align}\label{eq:first-L0-L1-0}
    f(\y) \leq f(\x) + \inner{\nabla f(\x)}{\vy-\vx} + \frac{AL_0 + BL_1\norm{\nabla f(\vx)}}{2}\norm{\vy-\vx}^2
\end{align}
for all $\norm{\vx-\vy}\leq c/L_1$, where $A,B,c$ are some constants. In fact we can strengthen (\ref{eq:first-L0-L1-0}) as follows
\begin{align*}
	&	|f(\y) - f(\x) - \dotprod{\nabla f(\x), \vy-\vx}| \\
= & \left|\int_{0}^1 \inner{\nabla f(\theta\vy+(1-\theta)\vx) -\nabla f(\vx)}{\vx-\vy}\,{\rm d}\theta\right| \\
\leq & \int_{0}^1 \norm{\nabla f(\theta\vy+(1-\theta)\vx) -\nabla f(\vx)}\norm{\vx-\vy}\,{\rm d}\theta \\
\leq & \int_{0}^1  (AL_0 + BL_1\norm{\nabla f(\x)})\norm{\theta\y-\theta\x}\norm{\x-\y}\,{\rm d}\theta \\
\leq & \int_{0}^1  \theta(AL_0 + BL_1\norm{\nabla f(\x)})\norm{\x-\y}^2\,{\rm d}\theta \\
\leq & \frac{AL_0 + BL_1 \norm{\nabla f(\x)}}{2} \norm{\y - \x}^2.
\end{align*}
It is natural to generalize the last term on above equation to the $\bH$-norm, leading to the inequality
\begin{equation*}
		|f(\y) - f(\x) - \dotprod{\nabla f(\x), \y-\x}| \leq \frac{AL_0 + BL_1 \norm{\nabla f(\x)}}{2} \norm{\y - \x}_{\bH}^2,
\end{equation*}
which is our Assumption \ref{ass:LL} if we ignore the difference in constants.}

\end{document}